\documentclass[a4paper,10pt]{article}
\usepackage{amsfonts}
\usepackage{color}
\usepackage{caption}
\usepackage{amssymb,amsmath,amsthm}
\usepackage{mathabx}
\usepackage{algorithmicx,algorithm}
\usepackage{graphicx}
\newtheorem{theorem}{Theorem}[section]
\newtheorem{corollary}[theorem]{Corollary}
\newtheorem{proposition}[theorem]{Proposition}
\newtheorem{lemma}[theorem]{Lemma}
\newtheorem{example}[theorem]{Example}
\newtheorem{definition}[theorem]{Definition}
\newtheorem{remark}[theorem]{Remark}
\newtheorem{assumption}[theorem]{Assumption}

\def\substack#1#2{{\scriptstyle{#1}\atop\scriptstyle{#2}}}

\def\tilde{\widetilde}
\def\Bar{\overline}
\def\ra{\rangle}
\def\la{\langle}

\def\ox{\bar{x}}

\def\emp{\emptyset}

\def\ph{\varphi}

\def\dom{\mbox{\rm dom}\,}

\def\Bar{\overline}

 \title{
 Generalized Metric Subregularity with Applications to\\ High-Order Regularized Newton Methods}
 \author{Guoyin Li\footnote{University of New South Wales, Sydney 2052, Australia (g.li@unsw.edu.au). Research of this author was partly supported by the Australian Research Council under Discovery Project DP190100555.} \and Boris Mordukhovich\footnote{Wayne State University, Detroit, MI 48202, USA (aa1086@wayne.edu). Research of this author was partly supported by the  US National Science Foundation under grant DMS-2204519, by the Australian Research Council under Discovery Project DP190100555, and by Project~111 of China under grant D21024.}\and Jiangxing Zhu\footnote{Yunnan University, Kunming 650091, China
(jxzhu@ynu.edu.cn). The research of this author was partly supported by the Basic Research Program of Yunnan Province (Grant No.
202201AT070066), the Project for Young-notch Talents in the Ten Thousand Talent Program of Yunnan Province (Grant No. YNWR-QNBJ-2020-080), and the National Natural Science Foundation of People’s Republic of China (Grant Nos. 12171419 and 12261109).}}
\begin{document}
\maketitle\vspace*{-0.25in}
\begin{abstract}

\noindent This paper pursues a twofold goal. First, we introduce and study in detail a new notion of variational analysis called {\em generalized metric subregularity}, which is a far-going extension of the conventional metric subregularity conditions.  Our primary focus is on examining this concept concerning first-order and second-order stationary points. We develop an extended convergence framework that enables us to derive superlinear and quadratic convergence under the generalized metric subregularity condition, broadening the widely used KL convergence analysis framework. We present verifiable sufficient conditions to ensure the proposed generalized metric subregularity condition and provide examples demonstrating that the derived convergence rates are sharp. Second, we design a new high-order regularized Newton method with momentum steps, and apply the generalized metric subregularity to establish its superlinear convergence. Quadratic convergence is obtained under additional assumptions. Specifically, when applying the proposed method to solve the (nonconvex) over-parameterized compressed sensing model, we achieve global convergence with a quadratic local convergence rate towards a global minimizer under a strict complementarity condition. \\[1ex]
\noindent{\small {\bf Keywords}: Variational analysis and optimization, generalized metric subregularity, error bounds, high-order regularized Newton methods, Kurdyka-\L ojasiewicz property, superlinear and quadratic convergence\\[1ex]
}
\end{abstract}\vspace*{-0.15in}

\section{Introduction}\label{intro}\vspace*{-0.05in}

It has been well recognized, especially during the recent years, that concepts and tools of modern variational analysis play a crucial role in the design and justification of numerical algorithms in optimization and related areas. We mention, in particular, {\em generalized Newton methods} for which, in addition to the fundamental developments summarized in \cite{facc-pang,IS,kk}, new approaches and results have been recently proposed in \cite{AAFV,GO,kmpt,li-sun-toh,m24} among other publications. A common feature of these developments is the usage of {\em metric regularity/subregularity} assumptions in a certain form. Another common feature of the aforementioned publications is justifying superlinear convergence rates under some {\em nondegeneracy condition}. For example, \cite{np06} develops the cubic regularization (CR) method for minimizing a twice continuously differentiable function, which is now widely recognized as a globally convergent variant of Newton’s method with superior iteration complexity \cite{CGT11,np06}. The local quadratic convergence of CR method was derived under a nondegeneracy condition in \cite{np06}. In \cite{yzs}, the authors propose an {\em error bound} condition related to the second-order stationary set and examined the quadratic local convergence for the cubic regularized Newton method under this condition, extending the results in \cite{np06}. Interestingly, it was shown in \cite{yzs} that this error bound condition can be applied to nonconvex and degenerate cases being satisfied with overwhelming probability for some concrete nonconvex optimization problems that arise in phase retrieval and low-rank matrix recovery.

On the other hand, it is known that the metric subregularity of the subdifferential mapping has a close relationship with the {\em  Kurdyka-\L ojasiewicz $(KL)$ property}. The KL property is satisfied for various important classes of functions arising in diverse applications, and it has emerged as one of the most widely used tools for analyzing the convergence of important numerical algorithms \cite{ab09,abs}. In particular, if the potential function of the algorithm satisfies the KL property and the associated desingularization function $\vartheta$ involved in the KL  property takes the form that $\vartheta(t) = c \, t^{1-\theta}$ with $c>0$ and $\theta \in [0,1)$, then the corresponding exponent $\theta$ determines the linear or sublinear convergence rate of the underlying algorithms \cite{ab09,abs,lp,YLP22}. Although the KL property and its associated analysis framework are well-adopted in the literature,  this framework, in general, cannot obtain or detect superlinear or quadratic 
convergence rate, which is typical for many second-order numerical methods including Newton-type methods.

In this paper, we introduce and thoroughly investigate novel {\em generalized metric subregularity conditions}, significantly extending conventional metric subregularity and error bound conditions from \cite{yzs}. We develop an abstract extended convergence framework that enables us to derive superlinear and quadratic convergence towards specific target sets, such as first-order and second-order stationary points, under the introduced generalized metric subregularity condition. This new framework extends the widely used KL convergence analysis framework. We provide examples demonstrating that the derived quadratic convergence rates are sharp, meaning they can be attained. We also present verifiable sufficient conditions that ensure the proposed generalized metric subregularity condition and showcase that the sufficient conditions are satisfied for functions that arise in practical models including the over-parameterized compressed sensing, best rank-one matrix approximation, and generalized phase retrieval problems.  Finally, we use the generalized metric subregularity condition to establish {\em superlinear} (and in some cases {\em quadratic}) convergence of the newly proposed {\em high-order regularized Newton methods with momentum}. Specifically, when applying the proposed method to solve the (nonconvex) over-parameterized compressed sensing model, we achieve global convergence with a quadratic local convergence rate towards a global minimizer under a strict complementarity condition. \vspace*{0.03in}

{\bf Organization.}  The main achievements of the current research are summarized in the concluding Section~\ref{conc}. The rest of the paper is structured as follows.

Section~\ref{prel} presents basic definitions and {\em preliminaries} from variational analysis and generalized differentiation broadly used in what follows. In Section~\ref{sec:subreg}, we define two versions of {\em generalized metric subregularity} for subgradient mappings with respect to first-order and second-order stationary points and then discuss their relationships with known in the literature and with {\em growth conditions}.

Section~\ref{sec:abst} develops a new framework for {\em abstract convergence analysis} to establish superlinear and quadratic convergence rates of numerical algorithms under generalized metric subregularity. Being of their independent interest, the obtained results are instrumental to justify such fast convergent rates for our main high-order regularized Newtonian algorithm in the subsequent parts of the paper. In Section~\ref{sec:kl}, we provide some abstract convergence analysis under the {\em Kurdyka-\L ojasiewicz property}. 

Section~\ref{sec:2nd} revisits generalized metric subregularity while now concentrating on this property with respect to {\em second-order stationarity points}. Our special attention is paid to the {\em strict saddle point} and {\em weak separation} properties, which provide verifiable sufficient conditions for the fulfillment of generalized metric subregularity. Here we present several examples of {\em practical modeling}, where these properties are satisfied.

In Section~\ref{sec:alg}, we design the family of {\em high-order regularized Newton methods with momentum steps} for problems of ${\cal C}^2$ optimization, where the order of the method is determined by the H\"older exponent $q\in(0,1]$ of the Hessian matrix associated with the objective function. Using the machinery of generalized metric subregularity with respect to second-order stationary points leads us to establishing {\em superlinear convergence} of iterates with explicit convergence rates and also {\em quadratic convergence} in certain particular settings.

Section~\ref{conc} provides an overview of the major results obtained in the paper with discussing some directions of the future research. Section~\ref{appe} is an {\em appendix}, which contains the proofs of some technical lemmas.\vspace*{-0.15in}
 
\section{Preliminaries from Variational Analysis}\label{prel}\vspace*{-0.05in}

Let $\mathbb{R}^m$ be the $m$-dimensional Euclidean space with the inner product denoted by $\langle x,y \rangle =x^Ty$ for all $x,y \in \mathbb{R}^m$ and the norm $\|\cdot\|$. The symbols $B_{\mathbb{R}^m}(x,\gamma)$ and $B_{\mathbb{R}^m}[x,\gamma]$ stand for the open and closed ball centered at $x$ with radius $\gamma$, respectively. Given a symmetric matrix $M\in \mathbb{R}^{m\times m}$, the notation $M\succeq 0$ signifies that $M$ is positive-semidefinite. We use $\mathbb{N}$ to denote the set of nonnegative integrals $\{0,1,\ldots\}$. 

An extended-real-valued function $f\colon\mathbb{R}^m\to\Bar{\mathbb{R}}:=(-\infty,\infty]$ is {\em proper} if ${\rm dom}\,f:=\{x\in\mathbb{R}^m\;|\;f(x)<\infty\}\ne\emp$ for its effective domain. The (limiting, Mordukhovich) {\em subdifferential} for such functions $f$ at $\ox\in\dom f$ is 
\begin{equation}\label{sub}
\begin{array}{ll}
\partial f(\ox):=\Big\{v\in\mathbb{R}^m\;\Big|&\exists\,x_k\to\ox\;\mbox{ with }\;f(x_k)\to f(\ox)\;\mbox{ as }\;k\to\infty\;\mbox{ with}\\
&\displaystyle\liminf_{x\to x_k}\frac{f(x)-f(x_k)-\la v,x-x_k\ra}{\|x-x_k\|}\ge 0\;\mbox{ for all }\;k\in\mathbb N\Big\}.
\end{array}
\end{equation}
The subdifferential \eqref{sub} reduces to the subdifferential of convex analysis when $f$ is convex and to the gradient $\nabla f(\ox)$ when $f$ is continuously differentiable (${\cal C}^1$-smooth) around $\ox$. For the general class of lower semicontinuous (l.s.c.) functions, $\partial f$ enjoys comprehensive calculus rules and is used in numerous applications of variational analysis and optimization; see, e.g., the books \cite{bs06,m24,rw} and the references therein.

Having \eqref{sub}, consider the set of {\em first-order stationary points} of an l.s.c.\ function $f\colon\mathbb{R}^m\to\Bar{\mathbb{R}}$ defined by
\begin{equation}\label{1stat}
\Gamma:=\big\{x\in\mathbb R^m\;\big|\;0 \in \partial f(x)\big\}.
\end{equation}
Suppose that $f$ is a twice continuously differentiable (${\cal C}^2$-smooth) function and denote the collection of {\em second-order stationary points} of $f$ at $x\in\mathbb{R}^m$ by
\begin{equation}\label{2stat}
\Theta:=\big\{x\in\mathbb R^m\;\big|\;\nabla f(x)=0,\;\nabla^2 f(x)\succeq 0\big\}.
\end{equation}
As in \cite{zz16}, we say that $\varphi:\mathbb{R}_+\rightarrow\mathbb{R}_+$, where $\mathbb{R}_+:=\{t\;|\;t\ge 0\}$, is an {\em admissible function} if $\varphi(0)=0$ and $\varphi(t)\rightarrow0\Rightarrow t\rightarrow0$. It is well known that for convex admissible functions, the directional derivative 
$$
\varphi_+'(t):=\lim\limits_{s\rightarrow0^+}\frac{\varphi(t+s)-\varphi(t)}{s}
$$
always exists being nondecreasing on $\mathbb{R}_+$. Furthermore (see, e.g., \cite[Theorem 2.1.5]{za}), $\varphi_+'$ is increasing on $\mathbb{R}_+$ if and only if $\varphi$ is strictly convex, i.e.,
$$
\varphi(\lambda t_1+(1-\lambda)t_2)<\lambda \varphi(t_1)+(1-\lambda)\varphi(t_2)
$$
for any $\lambda\in(0,1)$ and $t_1,t_2\in\mathbb{R}_+$ with $t_1\neq t_2$.
It is also known that a convex function $\varphi$ is differentiable on $\mathbb{R}_+$ if and only if $\varphi_+'$ is continuous on $\mathbb{R}_+$.
The following lemma from \cite[Theorem~2.1.5]{za} is used below.

\begin{lemma}[\bf properties of convex admissible functions]\label{phid}
Let $\varphi:\mathbb{R}_+\to \mathbb{R}_+$ be a convex admissible function. Then we have the relationships:\\
{\bf(i)} $0<\frac{\varphi(t_1)}{t_1}\leq \varphi_+'(t_1)\leq\varphi_+'(t_2)\;\;{\rm for\;all}\;t_1,t_2\in(0,\;\infty)\;{\rm with}\;t_1\leq t_2.$\\
{\bf(ii)} $ \varphi(t)\geq \varphi_+'\left(\frac{t}{2}\right)\frac{t}{2}\;\mbox{ for all }\;\in (0,\infty).$
\end{lemma}

Finally in this section, we recall the celebrated property often used in convergence analysis.

\begin{definition}[\bf Kurdyka-{\L}ojasiewicz  property]\label{de1}
{\rm Let $f:\mathbb{R}^m\to\Bar{\mathbb{R}}$ be a proper lower l.s.c.\ function, and let $\vartheta:[0, \eta)\to \mathbb{R}_+$ be a continuous concave function. We say that $f$ has the {\em Kurdyka-{\L}ojasiewicz $(KL)$ property} at $\overline{x}$ with respect to $\vartheta$ if the following conditions hold:\\
{\bf(i)} $\vartheta(0)=0$ and $\vartheta$ is continuously differentiable on $(0, \eta)$.\\
{\bf(ii)} $\vartheta'(s)>0$ for all $s\in (0, \eta)$.\\
{\bf(iii)} There exists $\varepsilon>0$  such that
\begin{equation*}
\vartheta'(f(x)-f(\overline{x}))d(0, \partial f(x))\geq 1
\end{equation*}
for all $x\in B_{\mathbb{R}^m}(\overline{x}, \varepsilon)\cap [f(\overline{x})<f<f(\overline{x})+\eta]$, where $d(\cdot,S)$ stands for the {\em distance function} associated with the set $S$. If in addition $\vartheta$ takes the form of $\vartheta(t)=c \,  t^{1-\theta}$ for some $c>0$ and $\theta\in[0,1)$, then we say $f$ satisfies the {\em KL property} at $\bar x$ with the {\em KL exponent $\theta$}.} 
\end{definition}

A widely used class of functions satisfying the KL property consists of {\em subanalytic functions}.  It is known that the maximum and minimum of finitely many analytic functions and also semialgebraic functions (such as $\|x\|_p=[\sum_{i=1}^m |x_i|^p]^{\frac{1}{p}}$ with $p$ being a nonnegative rational number) are subanalytic. We refer the reader to \cite{ab09,abs,BM,lp} for more details on the KL property and its exponent version, as well as on subanalytic functions.\vspace*{-0.1in}  

\section{Generalized Metric Subregularity}\label{sec:subreg}
\setcounter{equation}{0}\vspace*{-0.05in}

The main variational concept for the subdifferential \eqref{sub} of an extended-real-valued function used in this paper is introduced below, where we consider its two versions.

\begin{definition}[\bf generalized metric subregularity]\label{def3.1}
{\em Let $f:\mathbb{R}^m\to\Bar{\mathbb{R}}$ be a proper l.s.c.\ function, let $\Gamma$ be taken from \eqref{1stat}, and let $\psi:\mathbb{R}_+\to \mathbb{R}_+$ be an admissible function. Given a target set $\Xi \subset \Gamma$ and a point  $\overline{x}\in \Xi$, we say that:\\
{\bf(i)} The subdifferential  $\partial f$ satisfies the ({\rm pointwise}) {\em generalized metric subregularity property}  with respect to $(\psi, \Xi)$  at 
$\overline{x}$ if there exist $\kappa,\delta\in (0,\infty)$ such that
\begin{equation}\label{1ms}
\psi\big( d(x,\Xi)\big) \leq \kappa \, d\big(0, \partial f(x)\big)\;\mbox{ for all }\;x\in B_{\mathbb{R}^m}(\overline{x}, \delta).
\end{equation}
{\bf(ii)} The subdifferential $\partial f$ satisfies the {\em uniform generalized metric subregularity property} with respect to $(\psi,\Xi)$ if there exist $ \kappa, \rho\in (0,\infty)$ such that
\begin{equation}\label{2ms}
\psi\big( d(x,\Xi)\big) \leq  \kappa \, d\big(0, \partial f(x)\big)\;\mbox{ for all }\;x\in \mathcal{N}(\Xi, \rho):=\big\{x\in \mathbb{R}^m\;\big|\;\;d(x,\Xi )\le\rho\big\}.
\end{equation}
{\bf(iii)} If $\psi(t)=t$ in {\rm(i)} and {\rm(ii)}, then we simply say that $\partial f$ satisfies the {\em metric subregularity} (resp.\ {\em uniform metric subregularity}) property with respect to $\Xi$. }
\end{definition}

\begin{remark}[\bf connections with related notions]\label{rem1} {\rm Let us briefly discuss connections of the generalized metric subregularity notions from Definition~{\rm\ref{def3.1}} with some related notions.

$\bullet$ Consider the case where the target set $\Xi:=\Gamma$ is the set of {\em first-order stationary points} \eqref{1stat}. If $\psi(t)=t$, then \eqref{1ms} reduces to the conventional notion of {\em metric subregularity} applied to subgradient mappings. If $\psi(t)=t^p$ for $p>0$, which we referred to as {\em exponent metric subregularity}, this notion has also received a lot of attention due to the important applications in numerical optimization. It is called the  {\em H\"{o}lder metric subregularity} if $p>1$, and the {\em high-order metric subregularity} if $p \in (0,1)$. Other cases of convex and smooth admissible functions are considered in \cite{zz16}. The refer the reader to \cite{AG14,G,LM, MOU,myzz,ZNZ,zz16} and the bibliographies therein for more details.

$\bullet$ In the case where $f$ is a ${\cal C}^2$-smooth function, $\psi(t)=t$, and the target set $\Xi=\Theta$ consists of {\em second-order stationary points} \eqref{2stat}, Definition~\ref{def3.1}(ii) takes the form: there exist $ \kappa, \rho\in (0,\infty)$ such that
\begin{equation}\label{eb}
d(x,\Theta) \leq  \kappa \|\nabla f(x)\|\;\mbox{ for all }\;x\in \mathcal{N}(\Theta, \rho):=\big\{x\in \mathbb{R}^m\;\big|\;\;d(x, \Theta)\leq \rho\big\},
\end{equation}
i.e., it falls into the framework of the {\em uniform} generalized metric subregularity \eqref{2ms} for smooth functions. In this case, it was referred to as the {\em error bound} (EB) {\em condition} \cite[Assumption 2]{yzs}. This condition plays a key role in \cite{yzs} to justify the quadratic convergence of the cubic regularized Newton method.}
\end{remark}

Now we present three examples illustrating several remarkable features of the introduced notions. The first example shows that the generalized metric subregularity from \eqref{1ms} can go beyond the exponent cases. Note that the usage of metric subregularity in the nonexponent setting has been recently identified in \cite{llp} for the case of exponential cone programming.  

\begin{example}[\bf nonexponent generalized  metric subregularity]\label{exa:nonexp}
{\em Let $\varphi:\mathbb{R}_+\to \mathbb{R}_+$ be given by
\begin{equation*}
\varphi(t):=\left\{\begin{array}{ll}
\int_0^t e^{-\frac{1}{s}}ds, & \hbox{$t>0$,} \\
0, & \hbox{$t=0$}.
\end{array}\right.
\end{equation*}
Define $f:\mathbb{R} \to \mathbb{R}$ by $f(x):=\varphi(|x|)$ for all $x \in \mathbb{R}$. It is easy to see that $\varphi$ is a differentiable convex function on $\mathbb{R}_+$, and that we have
\begin{equation*}
\partial f(x)=\left\{
\begin{array}{ll}
\varphi'(|x|)\frac{x}{|x|}, & \hbox{$x\neq 0$,} \\
0, & \hbox{$x$}=0
\end{array}
\right.
\quad \mbox{and}\quad  \mathop{\arg\min}_{x\in \mathbb{R}}f(x)=\Gamma=\{0\}.
\end{equation*}
This tells us that
$\varphi'(d(x,  \Gamma))=\varphi'(|x|)=d(0, \partial f(x))$ for all $x \in \mathbb{R}$, and so $\partial f$ has the generalized metric subregularity property with respect to $(\varphi',\Gamma)$. On the other hand, for any $q\in (0,\infty)$ we get
\begin{equation*}
\lim_{x\to 0}\frac{d(x,\Gamma)^q}{d(0, \partial f(x))}=\lim_{x \to 0}\frac{|x|^q}{e^{-\frac{1}{|x|}}}=\lim_{t\to 0^+}e^\frac{1}{t}t^{q}=\infty,
\end{equation*}
which shows that $\partial f$ does not enjoy any exponent-type metric subregularity.}
\end{example}

The next example illustrates the fulfillment of the generalized metric subregularity \eqref{2ms} with respect of $(\psi,\Theta))$ for some admissible function $\psi$, while the error bound  condition \eqref{eb} from \cite{yzs} fails.

\begin{example}[\bf generalized metric subregularity vs.\ error bound condition]\label{ex:eb} {\rm Given $p>2$, define a ${\cal C}^2$-smooth function $f:\mathbb{R}\to \mathbb{R}$ by
\begin{equation*}
 f(x):=\left\{
\begin{array}{ll}
x^p, & \hbox{$x\geq 0$,} \\
0, & \hbox{$x<0$}
\end{array}
\right.
\end{equation*}
for which we get by the direct calculations that
\begin{equation*}
\nabla f(x)=\left\{
\begin{array}{ll}
px^{p-1}, & \hbox{$x\geq 0$,} \\
0, & \hbox{$x<0$,}
\end{array}
\right. \; \nabla^2 f(x)=\left\{
\begin{array}{ll}
p(p-1)x^{p-2}, & \hbox{$x\geq 0$,} \\
0, & \hbox{$x<0$},
\end{array}
\right.\; \mbox{and}\;\; \Theta=\Gamma=(-\infty, 0].
\end{equation*}
It is easy to see that  $ d(x, \Theta)=x$ and $d(0, \nabla f(x))=px^{p-1}$ for $x>0$, while for $x \le 0$ we have $ d(x, \Theta)=d(0, \nabla f(x))=0$. This implies therefore that
\begin{equation}\nonumber
\limsup_{x\to 0^+}\frac{d(x, \Theta)}{d(0, \nabla f(x))}=\infty\quad \mbox{and}\quad d (x, \Theta)^{p-1}\leq \frac{1}{p} d(0, \nabla f(x))\quad\mbox{whenever }\;x\in \mathbb{R}.
\end{equation}
Consequently, the EB condition \eqref{eb} fails, while $\nabla f$ enjoys the generalized metric subregularity property  with respect to $(\psi, \Theta)$  at $0$, where $\psi(t):=t^{p-1}$.}
\end{example}

For a fixed admissible function $\psi$ and a given set $\Xi \subset \Gamma$, we observe directly from the definitions that uniform generalized  metric subregularity with respect to $(\psi,\Xi)$ yields the generalized metric subregularity at each point $\overline{x}$ of $\Xi$. The following simple one-dimensional example shows that the reverse implication fails.

\begin{example}[{\bf pointwise version is strictly weaker than uniform version for generalized  metric subregularity}]\label{example:3.3}
{\rm Consider $f(x):=(x-a)^{2}(x-b)^{2p}$ for $p>1$ with $a,b \in \mathbb{R}$ and $0 \le a<b$. We get $\nabla f(x)=f'(x)=2(x-a)(x-b)^{2p}+2p (x-a)^2(x-b)^{2p-1}$ and $f''(x)=2(x-b)^{2p}+8p(x-a)(x-b)^{2p-1}+2p(2p-1)(x-a)^2(x-b)^{2p-2}$. Furthermore, $\Gamma=\{a,b,\frac{b+pa}{1+p}\}$ and $\Theta=\{a,b\}$. Thus there exists $M>0$ such that
$d(x,\Theta)=|x-a| \le M d(0, \nabla f (x))$ for all $x \in B_{\mathbb{R}}(a,\epsilon)$ with $\epsilon<\frac{b-a}{4}$. This tells us that $\nabla f$ satisfies the generalized  metric subregularity property with respect to $(\psi,\Xi)$ at $\overline{x}=a$ with $\psi(t)=t$ and $\Xi=\Theta$.

On the other hand, $\nabla f$ does not satisfy the uniform generalized metric subregularity with respect to $(\psi,\Xi)$. Indeed, for any $\epsilon_k>0$ with $\epsilon_k \rightarrow 0$, we get $x_k=b+\epsilon_k \in \mathcal{N}(\Theta,\epsilon_k)$ and
$\frac{d(x_k,\Theta)}{\|\nabla f(x_k)\|} = O(\frac{\epsilon_k}{\epsilon_k^{2p-1}}) \rightarrow \infty$.}
\end{example}

It has been well recognized in variational analysis that there exist close relationships between metric subregularity of subgradient mappings and {\em second-order growth conditions} for the functions in question; see \cite{AG14,dmn,m24} with more details. We now extend  such relationships to the case of generalized metric subregularity. The following useful lemma is of its own interest.

{\begin{lemma}[\bf first-order stationary points of compositions] \label{LE2.1}
Let $f(x)=(\varphi\circ g)(x)$, where $\varphi: \mathbb{R}^n \to \Bar{\mathbb{R}}$ is a proper l.s.c.\ convex function, and where $g : \mathbb{R}^m \to \mathbb{R}^n$ is a ${\cal C}^1$-smooth mapping  with the surjective derivative $\nabla g(\overline{x})$ at some first-order stationary point $\overline{x} \in\Gamma$ associated with the composition $f$ in \eqref{1stat}. Then there exists $\delta>0$ such that we have the inclusion
\begin{equation}\label{217}
\Gamma\cap B_{\mathbb{R}^m}(\overline{x}, \delta)\subseteq \mathop{\arg\min}_{x\in \mathbb{R}^m}f(x).
\end{equation}
\end{lemma}
\begin{proof}
The surjectivity of $\nabla g(\overline{x})$ for smooth $g$ allows us to deduce from the Lyusternik–Graves theorem (see, e.g., \cite[Theorem~1.57]{bs06}) the existence of $\delta, L>0$ such that $\nabla g(x)$ is surjective on $B_{\mathbb{R}^m}(\overline{x}, \delta)$ and
\begin{equation}\label{l2.10}
\|\nabla g(x)^T v\| \ge L\|v\|\quad\mbox{whenever }\;(x,v)\in B_{\mathbb{R}^m}(\overline{x},\delta)\times \mathbb{R}^n.
\end{equation}
Consequently, it follows from \cite[Proposition~1.112]{bs06} that
\begin{equation}\label{3.58}
\partial f(x)=\nabla g(x)^T\partial \varphi(g(x))\quad\mbox{for all }\;x\in B_{\mathbb{R}^m}(\overline{x}, \delta).
\end{equation}
Picking any $u\in \Gamma\cap B_{\mathbb{R}^m}(\overline{x}, \delta)$, we deduce from \eqref{3.58} that
\begin{equation}\nonumber
0\in \partial f(u)=\nabla g(u)^T\partial \varphi(g(u)).
\end{equation}
Combining this with \eqref{l2.10} yields $0\in \partial \varphi(g(u))$. Hence $g(u)$ is a minimizer of the convex function $\varphi$, and 
\begin{equation}\nonumber
f(u)=\varphi(g(u))\leq \varphi(g(x))=f(x)\;\mbox{ for all}\;x\in \mathbb{R}^m,
\end{equation}
which gives us \eqref{217} and thus completes the proof.
\end{proof}}

Now we are ready to establish relationships between the generalized metric subregularity and corresponding growth conditions extending the known ones \cite{AG14,dmn} when $\psi(t):=t^2$, $\overline{\psi}(t):=t$ and 
$\Xi=\Gamma$. 

\begin{theorem}[\bf generalized metric subregularity via growth conditions]\label{thm:gro} Let $f(x)=(\varphi\circ g)(x)$ in the setting of Lemma~{\rm\ref{LE2.1}}. Consider the following statements:\\
{\bf (i) } There exists an an admissible function $\psi:\mathbb{R}_+\to \mathbb{R}_+$ and constants $\gamma_1, \kappa_1 \in (0,\infty)$  such that
\begin{equation}\label{210}
f(x)\geq f(\overline{x})+\kappa_1 {\psi}( d(x, \Xi))\quad\mbox{for all }\;x\in B_{\mathbb{R}^m}(\overline{x}, \gamma_1).
\end{equation}
{\bf(ii)} There exists an an admissible function $\overline \psi:\mathbb{R}_+\to \mathbb{R}_+$ and constants $\gamma_2, \kappa_2\in (0,\infty)$ such that
\begin{equation}\nonumber
\overline \psi(d(x,\Xi)) \leq  \kappa_2 d(0, \partial f(x))\quad\mbox{for all x}\;\in B_{\mathbb{R}^m}(\overline{x},\gamma_2).
\end{equation}
Then implication ${\rm(i)}\Rightarrow{\rm(ii)}$ holds with $\overline\psi(t):= {\psi}_+'(\frac{t}{2})$ if $\psi$ is convex and  $\lim_{t\to 0^+}\psi_+'(t)=\psi_+'(0)=0$, while the reverse one ${\rm(ii)}\Rightarrow{\rm(i)}$ is always satisfied with $\psi(t):= \int_0^t \,  \overline{\psi}(s)ds$.
\end{theorem}
\begin{proof}  To verify ${\rm(i)}\Rightarrow{\rm(ii)}$, we proceed as in the proof of Lemma~\ref{LE2.1} and find $\delta, L>0$ such that \eqref{217}, \eqref{l2.10}, and \eqref{3.58} hold. Letting $\tilde \delta:=\min\{\gamma_1, \frac{\delta}{2}\}$ and then picking any $x\in B_{\mathbb{R}^m}(\overline{x}, \tilde \delta)$ and $u\in \partial f(x)$, take a sequence $\{x_k\}\subset \Xi$ such that
$\|x-x_k\|\to d(x, \Xi)$ as $k\to\infty$. Noting that $\overline{x}\in \Xi$ and $\|x-\overline{x}\|<\tilde \delta \leq \frac{\delta}{2}$, assume without loss of generality that $\{x_k\}$ entirely lies in $\Xi\cap B_{\mathbb{R}^m}(\overline{x}, 2\tilde \delta)$. Select further $L_0>0$ with
$ \|g(x_2)-g(x_1)\|\leq L_0\|x_1-x_2\|$ for all $x_1, x_2\in B_{\mathbb{R}^m}(\overline{x}, \delta)$ and show that
\begin{equation}\label{3.60}
f(x_{k})\geq f(x)-\frac{L_0}{L}\|u\|\cdot \|x_{k}-x\|\;\mbox{ for all }\;k\in\mathbb N
\end{equation}
whenever $u\in\partial f(x)$. Indeed, it follows from \eqref{l2.10} and \eqref{3.58} that for any $u\in\partial f(x)$ there is $v\in \partial \varphi(g(x))$ with
\begin{equation*}
 u=\nabla g(x)^T v\quad \mbox{and}\quad L\|v\|\le\|u\|.
\end{equation*}
Combining the latter with the convexity of $\varphi$ gives us
\begin{equation*}
\begin{aligned}
&f(x)-f(x_k)=\varphi(g(x))-\varphi(g(x_k))\le\langle v, g(x)-g(x_k)\rangle\\
&\le\|v\| \cdot \|g(x)-g(x_k)\|\le L_0\|v\|\cdot \|x-x_k\|\le\frac{L_0}{L}\|u\|\cdot\|x-x_k\|
\end{aligned}
\end{equation*}
for all $k\in\mathbb N$, and thus brings us to \eqref{3.60}. Note further that \eqref{217} ensures that $f(x_k)=f(\overline{x})$ for all $k\in\mathbb N$. Letting there $k\to\infty$ and denoting $\ell:=\frac{L_0}{L}$, we get $f(x)-f(\overline{x})\leq \ell\|u\| d(x, \Xi)$, which yields in turn 
\begin{equation*}
f(x)-f(\overline{x})\leq \ell d\left(0, \partial f(x)\right) d(x, \Xi)\;\mbox{ whenever }\;x\in B_{\mathbb{R}^m}(\overline{x}, \tilde \delta).
\end{equation*}
This allows us to derive from \eqref{210} that
\begin{equation*}
0<\frac{\kappa_1}{\ell}\frac{ \psi \left( d(x, \Xi)\right)}{ d(x, \Xi)}\leq d\left(0, \partial f(x)\right)\;\mbox{ for all }\;x\in B_{\mathbb{R}^m}\left(\overline{x}, \tilde\delta\right)\setminus {\rm cl}(\Xi),
\end{equation*}
where ${\rm cl}(\Xi)$ stands the closure of $\Xi$. Employing 
Lemma~\ref{phid}(ii) tells us that
\begin{equation*}
\frac{\kappa_1}{2\ell}\psi_+'\left(\frac{d(x, \Xi)}{2}\right)\leq d\left(0, \partial f(x)\right)
\end{equation*}
for any $x\in B_{\mathbb{R}^m}\left(\overline{x}, \tilde \delta\right)\setminus {\rm cl}(\Xi)$, and thus it follows from  $\psi_+'(0)=0$ that (ii) is satisfied. The reverse implication (ii)$\Rightarrow$(i) is a consequence of 
Lemma~\ref{LE2.1} and \cite[Theorem~4.2]{yzz}.
\end{proof}\vspace*{-0.2in}

\section{Abstract Framework for Convergence Analysis} 
\setcounter{equation}{0}\label{sec:abst}\vspace*{-0.05in}

In this section, we develop a general framework for convergence analysis of an abstract minimization algorithm to ensure its {\it superlinear} and {\em quadratic convergence}. This means finding appropriate conditions on an abstract algorithm structure and the given objective function under which the generated iterates exhibit the desired fast convergence. We'll see that the generalized metric subregularity is an important condition to establish fast convergence. The obtained abstract results are implemented in Section~\ref{sec:alg} to derive specific convergence conditions for the newly proposed high-order regularized Newton method with momentum. \vspace*{0.03in}

Our study in this section focuses on a coupled sequence $\{(x_k,e_k)\} \subset \mathbb{R}^m \times \mathbb{R}_+$, where $\{x_k\}$ is the main sequence generated by the algorithm, and where $\{e_k\}$ is an auxiliary numerical sequence that serves as a surrogate for $\|x_{k+1}-x_k\|$ (the magnitude of the successive change sequence). Throughout this section, we assume that these two sequences satisfy the following set of {\bf basic assumptions (BA)}:
\begin{itemize}
\item[{\bf(i)}] {\em Descent condition}:
\begin{equation}\tag{${\rm H0}$}
 f(x_{k+1})\le f(x_k)\quad\mbox{ for all }\;k\in \mathbb{N} \label{H1}.
\end{equation}

\item[{\bf(ii)}] {\em Surrogate condition}: there exists $c>0$ such that
\begin{equation}\tag{${\rm H1}$}
\| x_{k+1}-x_k\|\leq c \, e_k\;\mbox{ for all }\;k\in \mathbb{N}\;\mbox{ with }\;\lim_{k\to \infty}e_k=0. \label{H4}
\end{equation}

\item[{\bf(iii)}] {\em  Relative error condition}:
\begin{equation}\tag{${\rm H2}$}
\mbox{there exists }\;w_{k+1}\in \partial f(x_{k+1})\;\mbox{ such that }\;\|w_{k+1}\|\leq b \, \beta(e_k),\label{H2}
\end{equation}
where $b$ is a fixed positive constant, and where  $\beta:\mathbb{R}_+\to \mathbb{R}_+$ is an admissible function.

\item[{\bf(iv)}] {\em Continuity condition}:  For each subsequence $\{x_{k_j}\}$ with $\lim_{j\to \infty}x_{k_j}=\widetilde x$, we have
\begin{equation}\tag{${\rm H3}$}
\limsup_{j\to \infty}f(x_{k_j})\leq f(\widetilde x).\label{H3}
\end{equation}
\end{itemize}
\vspace{-0.1cm}
Note that the above descent condition (H0) and surrogate condition (H1) can  be enforced  by the following stronger property: there exist constants $a,c>0$ such that
\begin{equation}\tag{${\rm H1'}$}
f(x_{k+1})+a \, \varphi(e_k)\leq f(x_k)\;\mbox{ and }\;\| x_{k+1}-x_k\|\leq c \, e_k\;\mbox{ for all }\;k\in \mathbb{N}, \label{H001}
\end{equation}
where $\varphi$ is an admissible function. When $\varphi(t)=t^2$ and $e_k=\|x_{k+1}-x_k\|$, condition \eqref{H001} is often referred to as the {\em sufficient descent condition}. Both conditions \eqref{H001} and \eqref{H2} are automatically satisfied for various commonly used algorithms. For example, 
for proximal-type methods \cite{abs}, these conditions always hold with $\varphi(t)=t^2$, $\beta(t)=t$, $c=1$, and $e_k=\|x_{k+1}-x_k\|$; for cubic regularized Newton methods with Lipschitz continuous Hessian \cite{yzs}, the latter conditions are fulfilled with  $\varphi(t)=t^3$,  
$\beta(t)=t^2$, $c=1$, and $e_k=\|x_{k+1}-x_k\|$; for high-order proximal point methods \cite{AN24}, $\varphi(t)=t^{p+1}$ with $p \ge 1$, $\beta(t)=t^p$, $c=1$, and $e_k=\|x_{k+1}-x_k\|$.

Incorporating the surrogate sequence $\{e_k\}$ gives us some flexibility in handling multistep numerical methods with, e.g., momentum steps. As we see in Section~\ref{sec:alg}, conditions \eqref{H001} and \eqref{H2} hold with $e_k=\|\widehat{x}_{k+1}-x_k\|$, $\varphi(t)=t^{q+2}$, and $\beta(t)=t^{q+1}$ for some $q>0$, where $\widehat{x}_{k}$ is an auxiliary sequence generated by the subproblem of the high-order regularization methods, where $x_k={\rm argmin}_{x \in \{\widehat{x}_k,w_k\}}f(x)$, and where $w_k=\widehat{x}_k + \beta_k(\widehat{x}_k-\widehat{x}_{k-1})$ with $\beta_k  \ge 0$. In the special case where $\varphi(t)=t^2$ and $\beta(t)=t$, there are other general unified frameworks \cite{ochs}, which involve a bivariate function $F$ serving as an approximation for $f$. While we can potentially extend further in this direction, our main aim is to provide a framework applicable for superlinear and quadratic convergence, which is not discussed in \cite{ochs}.\vspace*{0.03in}

From now on, we assume that $f:\mathbb{R}^m\to\Bar{\mathbb{R}}$ is a {\em proper, l.s.c., bounded from below} function. For a given sequence $\{x_k\}$, use $\Omega$ to denote the set of its {\em accumulation points}. Given $x\in \mathbb{R}^m$, define $\mathcal{L}(f(x)):=\{y\in \mathbb{R}^m\;|\; f(y)\leq f(x)\}$. The next two lemmas present some useful facts needed in what follows.

\begin{lemma}[\bf accumulation points under basic assumptions]\label{P10}
Consider a sequence $\{(x_k,e_k)\}$ under assumption \eqref{H1}--\eqref{H3}, and let $\mathcal{L}(f(x_{k_0}))$ be bounded for some $k_0\in\mathbb N$. Then the following assertions hold:\\
{\bf(i)} The set of accumulation points $\Omega$ is nonempty and compact.\\
{\bf(ii)} The sequence $\{x_k\}$ satisfies the relationships
\begin{equation*}\label{2.1}
f(x)=\inf_{k\in\mathbb{N}}f(x_k)=\lim_{k\to \infty}f(x_k)\quad\mbox{for all }\;x\in \Omega.
\end{equation*}
{\bf(iii)} $\Omega\subset\Gamma$ and for any set $\Xi$ satisfying $\Omega\subset\Xi\subset\Gamma$, we have
\begin{equation}\label{5236}
\lim_{k\to \infty}d(x_k, \Xi)=0,
\end{equation}
where $\Gamma$ is the collection of the first-order stationary points \eqref{1stat}.
\end{lemma}
\begin{proof} To verify (i), observe that the descent condition \eqref{H1} and the boundedness of $\mathcal{L}(f(x_{k_0}))$ ensure the boundedness of $\{x_k\}$, which yields (i) since the closedness of $\Omega$ is easily derived by the diagonal process. 

To proceed with (ii), we deduce from \eqref{H1} that
\begin{equation}\label{5.0228}
 \inf_{k\in\mathbb{N}}f(x_k)=\lim_{k\to \infty}f(x_k).
\end{equation}
Taking any $x\in \Omega$ allows us to find a subsequence $\{x_{k_j}\}$ such that $\lim_{j\to \infty}x_{k_j}= x$. Combining this with condition \eqref{H3} and the l.s.c.\ of $f$ gives us
$\limsup_{j\to \infty}f(x_{k_j})\leq f(x)\leq \liminf_{j\to \infty}f(x_{k_j})$,
which establishes (ii)  based on the relationship provided in \eqref{5.0228}.

To verify the final assertion (iii), combine  (i), (ii), and condition \eqref{H2} to get $\Omega\subset \Gamma$.  For any set $\Xi$ satisfying $\Omega\subset \Xi\subset \Gamma$, \eqref{5236} is a direct consequence of the fact that
$\lim_{k\to \infty}d(x_k, \Omega)=0$.
\end{proof}

For a function $\tau: \mathbb{R}_+ \rightarrow \mathbb{R}_+$, we say it is {\em asymptotically shrinking} around zero  if $\tau(0)=0$ and
\begin{equation}\label{5228}
 \limsup_{t\to 0^+}\sum_{n=0}^{\infty}\frac{\tau^n(t)}{t}<\infty,
\end{equation}
where $\tau^0(t):=t$ and $\tau^n(t):=\tau(\tau^{n-1}(t))$ for all $n\in\mathbb N$. It follows from the definition, that the function $\tau(t)= \alpha t$ with $\alpha \in (0,1)$ is asymptotically shrinking around zero, while $\tau(t)=t$ is not. The next simple lemma provides an easily verifiable sufficient condition ensuring this property.

\begin{lemma}[\bf asymptotically shrinking functions]\label{L1.1}
Let $\tau:\mathbb{R}_+ \rightarrow \mathbb{R}_+$ be a nondecreasing function with $\tau(0)=0$ and such that
$\limsup_{t\to 0^+}\frac{\tau(t)}{t}<1$. Then this function is asymptotically shrinking around zero.
\end{lemma}
\begin{proof}
To verify \eqref{5228}, denote $r:=\limsup_{t\to 0^+}\frac{\tau(t)}{t}<1$  and for any $\varepsilon\in \left(r,  1\right)$ find $\delta>0$ such that
$\tau(t)<\varepsilon t$ whenever $t\in (0, \delta)$, which yields
$\tau^2(t)=\tau(\tau(t))<\varepsilon \tau(t)<\varepsilon^2 t$ as $t\in (0, \delta)$. By induction, we get $\tau^n(t)\leq \varepsilon^n t$ for all $n\in \mathbb{N}\; \mbox{and} \; t\in (0, \delta)$. Hence
\begin{equation*}
\sum_{n=0}^{\infty}\frac{\tau^n(t)}{t}\leq \sum_{n=0}^{\infty}\varepsilon^n=\frac{1}{1-\varepsilon}\quad\;\mbox{whenever }\; t\in (0, \delta),
\end{equation*}
which justifies \eqref{5228} and thus completes the proof.
\end{proof}

Now we are ready to establish our first abstract convergence result.

\begin{theorem}[{\bf abstract convergence framework}]\label{T5.3}
For some $\eta>0$, let $\xi:[0, \eta)\to \mathbb{R}_+$ be a nondecreasing continuous function with $\xi(0)=0$, and let $\{s_k\}$ be a nonnegative sequence  with $\lim_{k\to \infty}s_k=0$. Consider a sequence $\{(x_k,e_k)\}\subset \mathbb{R}^m \times \mathbb{R}_+$, which satisfies conditions \eqref{H1}--\eqref{H3}, and assume that $\mathcal{L}(f(x_{k_0}))$ is bounded for some $k_0\in\mathbb N$. Take further any point $\overline{x}$, a closed set $\Xi$ with $\overline{x}\in \Omega\subset \Xi \subset \Gamma$, and a nondecreasing function $\tau:\mathbb{R}_+ \to \mathbb{R}_+$ that is asymptotically shrinking around zero. Suppose that there exist $\ell_1 \in [0,1), \ell_2,\ell_3 \in [0,\infty)$, { and  $k_1 \in \mathbb{N}\setminus\{0\}$} such that  the following hold whenever $k \ge k_1$ with $x_{k}\in B_{\mathbb{R}^m}(\overline{x},\eta)$:
\begin{itemize}
\item[{\bf(i)}] {\bf The sequence $\{s_k\}$ shrinks with respect to the mapping $\tau$}, i.e.,
\begin{equation}\label{2029}
 \, s_k\leq  \, \tau (s_{k-1}).
 \end{equation}

\item[{\bf(ii)}] {\bf Surrogate of successive change grows mildly}, i.e.,   \begin{equation}\label{228}
 e_k\leq \ell_1e_{k-1}+ \ell_2\Lambda_{k, k+1}+ \ell_3 s_k,
\end{equation}
where the sequence $\{ \Lambda_{k, k+1}\}$ is defined by 
\begin{equation}\label{229}
 \Lambda_{k, k+1}:= \xi(f(x_k)-f(\overline{x}))-\xi(f(x_{k+1})-f(\overline{x})).
\end{equation}
\end{itemize}
Then there exists $k_2 \in \mathbb{N}$ with $k_2 \ge k_1$ such that  the inclusion $x_k\in B_{\mathbb{R}^m}(\overline{x}, \eta)$ holds 
for all $k > k_2$, and therefore the estimates in \eqref{2029} and \eqref{228} are also fulfilled for all $k> k_2$. Moreover, it follows that
\begin{equation}\label{abs-conver}
\sum_{k=0}^{\infty}e_k<\infty,\quad\sum_{k=1}^{\infty}\|x_{k+1}-x_k\|<\infty,
\end{equation}
and the sequence $\{x_k\}$ converges to $\overline{x} \in \Xi$ as $k\to \infty$.
\end{theorem}
\begin{proof} We first see that $\limsup_{k\to \infty}\sum_{n=0}^{\infty} \tau^n(s_k)=0$. To show this, we observe that $\sum_{n=0}^{\infty} \tau^n(0)=0$ and assume without loss of generality that $s_k\neq 0$ for all $k\in \mathbb{N}$. Since $\lim_{k\to \infty}s_k=0$ and $\tau$ is asymptotically shrinking around zero, we get the equalities
\begin{equation}\label{4.008}
\limsup_{k\to \infty}\sum_{n=0}^{\infty} \tau^n(s_k)=\limsup_{k\to \infty}\frac{\sum_{n=0}^{\infty} \tau^n(s_k)}{s_k}\cdot s_k=\lim_{k\to \infty}s_k\cdot\limsup_{k\to \infty}\sum_{n=0}^{\infty} \frac{\tau^n(s_k)}{s_k}=0.
\end{equation}
Recalling that $\xi$ is continuous, it follows from \eqref{4.008} combined with \eqref{H1}, \eqref{H4}, and Lemma~\ref{P10} that there exists $k_2\geq \max\{k_0, k_1\}$ such that $x_{k_2}\in B_{\mathbb{R}^m}(\overline{x}, \frac{\eta}{2})$, and for all $ k\geq k_2$ we have
 $f(\overline{x})\leq f(x_k)<f(\overline{x})+\eta$ and
\begin{equation}\label{235}
 \frac{\ell_1}{1- \ell_1}e_{k-1}+\frac{\ell_2}{1-\ell_1}\xi(f(x_{k})-f(\overline{x}))+\frac{\ell_3 }{1- \ell_1}\sum_{n=0}^{\infty}\tau^n(s_k)<\frac{\eta}{2c}.
\end{equation}
The rest of the proof is split into the following two claims.\vspace*{0.05in}

\underline{Claim~1}: {\em For every $k>k_2$, the inclusion $x_k\in B_{\mathbb{R}^m}(\overline{x}, \eta)$ holds.} Indeed, suppose on the contrary that there exists $k>k_2$ such that $x_k\notin B_{\mathbb{R}^m}(\overline{x}, \eta)$. Letting $\bar k:=\min\{k>k_2\;|\; x_k\notin B_{\mathbb{R}^m}(\overline{x}, \eta)\}$, observe that for any $k\in \{k_2, \dots, \bar k-1\}$, we have $x_k \in B_{\mathbb{R}^m}(\overline{x}, \eta)$, and thus the inequalities in \eqref{2029} and \eqref{228} are satisfied. Using them together with \eqref{229} implies that
\begin{equation*}
\begin{aligned}
\sum_{k=k_2}^{\bar k-1}e_k\leq & \ \ell_1 \sum_{k=k_2}^{\bar k-1}e_{k-1}+ \ell_2\sum_{k=k_2}^{\bar k-1}\Lambda_{k, k+1}+ \ell_3\sum_{k=k_2}^{\bar k-1} s_k\\
\leq &\ \ell_1 \sum_{k=k_2}^{\bar k-1}e_{k-1}+\ell_2 \, \xi\big(f(x_{k_2})-f(\overline{x})\big) +\ell_3\big(s_{k_2} +\tau(s_{k_2}) +\cdots+\tau^{\bar k-k_2-1} (s_{k_2})\big)\\
\leq &\ \ell_1 \sum_{k=k_2}^{\bar k-1}e_{k-1}+ \ell_2 \, \xi\big(f(x_{k_2})-f(\overline{x})\big) +\ell_3\sum_{n=0}^{\infty}\tau^n(s_{k_2}).
\end{aligned}
\end{equation*}
Rearranging the terms above brings us to the estimate
\begin{equation}\label{5240}
\begin{aligned}
\sum_{k=k_2}^{\bar k-1}e_k\leq &\frac{\ell_1}{1- \ell_1}e_{k_2-1}+\frac{\ell_2}{1- \ell_1}\xi\big(f(x_{k_2})-f(\overline{x})\big)+\frac{\ell_3}{1- \ell_1}\sum_{n=0}^{\infty}\tau^n(s_{k_2}).
\end{aligned}
\end{equation}
Combining this with \eqref{235} and condition \eqref{H4}, we arrive at
\begin{equation*}
\begin{aligned}
\|x_{\bar k}-\overline{x}\|\leq &\|x_{k_2}-\overline{x}\|+\sum_{k=k_2}^{\bar k-1}\| x_{k+1}-x_k\|
\leq \|x_{k_2}-\overline{x}\|+c\sum_{k=k_2}^{\bar k-1}e_k<\eta,
\end{aligned}
\end{equation*}
which contradicts the assumption that $x_{\bar k}\notin B_{\mathbb{R}^m}(\overline{x}, \eta)$ and thus justifies the claim.\vspace*{0.05in}

\underline{Claim~2}: {\em It holds that $\sum_{k=1}^{\infty}\|x_k-x_{k+1}\|<\infty$}. Indeed, Claim~1 tells us that the inequalities in \eqref{2029} and \eqref{228} are satisfied for all $k> k_2$. Consequently, for any $k\geq k_2$ and $p>k$, by similar arguments as those for deriving \eqref{5240} we get the estimate
\begin{equation}\label{5242}
\begin{aligned}
\sum_{i=k}^{p}e_i\leq & \frac{\ell_1}{1- \ell_1}e_{k-1}+\frac{\ell_2}{1- \ell_1}\xi\big(f(x_{k})-f(\overline{x})\big)+\frac{\ell_3 }{1- \ell_1}\sum_{n=0}^{\infty}\tau^n(s_k).
\end{aligned}
\end{equation}
Letting $p \rightarrow \infty$ and employing \eqref{235} yields
\begin{equation*}
\begin{aligned}
\sum_{k=0}^{\infty}e_k\leq &\sum_{k=0}^{k_2-1}e_k+\sum_{i=k_2}^{\infty}e_i\\
\leq &\sum_{k=0}^{k_2-1}e_k+ \frac{\ell_1}{1- \ell_1}e_{k_2-1}+\frac{\ell_2}{1- \ell_1}\xi\big(f(x_{k_2})-f(\overline{x})\big)+\frac{\ell_3}{1- \ell_1}\sum_{n=0}^{\infty}\tau^n(s_{k_2})<\infty,
\end{aligned}
\end{equation*}
which gives us $\sum_{k=1}^{\infty}\| x_{k+1}-x_k\|<\infty$ due to \eqref{H4}. Therefore, $\{x_k\}$ is a Cauchy sequence. Since $\overline{x}$ is an accumulation point of $\{x_k\}$, it follows   that  $x_k \rightarrow \overline{x} \in \Xi$  as $k\to \infty$.
\end{proof}

In what follows, we study the {\em convergence rate} of $\{x_k\}$.  For each $k$, denote $\Delta_k:=\sum_{i=k}^{\infty}e_i$ and deduce from Theorem~\ref{T5.3} that $\Delta_k<\infty$. For any $k$ sufficiently large, it follows from \eqref{5242} that 
\begin{equation*}\nonumber
\begin{aligned}
\Delta_k= \sum_{i=k}^{\infty}e_i\leq & \frac{\ell_1}{1- \ell_1}e_{k-1}+\frac{\ell_2}{1- \ell_1}\xi\big(f(x_{k})-f(\overline{x})\big)+\frac{\ell_3}{1- \ell_1}\sum_{n=0}^{\infty}\tau^n(s_k),
\end{aligned}
\end{equation*}
which ensures from \eqref{H4} that
\begin{equation}\label{256}
\begin{aligned}
\|\overline{x}-x_{k}\|=& \lim_{j\to \infty}\|x_{k+j}-x_{k}\| \leq   \sum_{i=k}^{\infty}\| x_{i+1}-x_i\|
\leq   \ c\sum_{i=k}^{\infty}e_i=c\Delta_k\\
\leq & \ \frac{c\,\ell_1}{1- \ell_1}e_{k-1}+\frac{c\,\ell_2}{1- \ell_1}\xi\big(f(x_{k})-f(\overline{x})\big) +\frac{c\,\ell_3}{1- \ell_1}\sum_{n=0}^{\infty}\tau^n(s_k).
\end{aligned}
\end{equation}

The next theorem establishes a fast convergence rate of the iterative sequence generated by the abstract algorithm satisfying the basic assumptions (BA) under generalized metric subregularity.

\begin{theorem}[{\bf convergence rate under generalized metric subregularity}]\label{P5.2} Let $\psi, \beta:\mathbb{R}_+\to \mathbb{R}_+$ be increasing admissible functions, let $\{(x_k, e_k)\}$ be a sequence satisfying the basic assumptions \eqref{H1}--\eqref{H3}, and let $\overline{x}\in\Omega$ be an accumulation point of $\{x_k\}$. Fixing any closed set $\Xi$ with $\Omega\subset \Xi\subset \Gamma$, assume that $\mathcal{L}(f(x_{k_0}))$ is bounded for some $k_0\in\mathbb N$ and the following properties are fulfilled:
\begin{itemize}
\item[{\bf(i)}] {\bf Generalized metric subregularity holds at $\overline{x}$ with respect to $(\psi,\Xi)$}, meaning that there exist numbers $\gamma, \eta\in (0,\infty)$ such that
\begin{equation}\label{0890}
\psi\big(d(x, \Xi)\big)\leq  \gamma \, d\big(0, \partial f(x)\big)\quad\mbox{ for all }\;x\in B_{\mathbb{R}^m}(\overline{x}, \eta).
\end{equation}
\item[{\bf(ii)}] {\bf Surrogate of successive change bounded in terms of the target set}, meaning that there exist $\ell>0$ and $k_1 \in \mathbb{N}$ such that
\begin{equation}\label{0900}
e_k\leq \ell \, d(x_k, \Xi)\quad\mbox{ for all }\;k \ge k_1.
\end{equation}
\item[{\bf(iii)}] {\bf Growth rate control is satisfied}, meaning that $\limsup_{t\to 0^+}\frac{\tau(t)}{t}<1$, where  $\tau:\mathbb{R}_+\to \mathbb{R}_+$ is given by $\tau(t):= \psi^{-1}\left(\gamma b\beta( \ell t)\right)$ for all $t\in \mathbb{R}_+$.
\end{itemize} 
Then the sequence $\{x_k\}$ converges to $\overline{x} \in \Xi$ as $k\to \infty$ with the rate $\|x_{k}-\overline{x}\|=O(\tau(\|x_{k-1}-\overline{x}\|))$, i.e.,
\begin{equation}\label{0910}
\limsup_{k\to \infty}\frac{\|x_{k}-\overline{x}\|}{{\tau}(\|x_{k-1}-\overline{x}\|)}<\infty.
\end{equation}
\end{theorem}
\begin{proof}
Let us first check that, for any $k \ge k_1+1$ with $x_k\in B_{\mathbb{R}^m}(\overline{x}, \eta)$, we have the estimate
\begin{equation}\label{0920}
 d(x_{k}, \Xi)\leq \tau\big( d(x_{k-1}, \Xi)\big).
\end{equation}
To this end,  fix $k \ge k_1+1$ such that  $x_{k}\in B_{\mathbb{R}^m}(\overline{x}, \eta)$ and deduce from \eqref{H2}, \eqref{0890}, and \eqref{0900} that
\begin{equation*}
\psi\big(d(x_{k}, \Xi)\big)\leq \gamma \, d\big(0, \partial f(x_k)\big) \le \gamma b \, \beta(e_{k-1})\leq \gamma b \, \beta\big(\ell \, d(x_{k-1}, \Xi)\big),
\end{equation*}
where the last inequality is due to the fact that $\beta$ is an increasing admissible function. Since $\psi$ is also an increasing admissible function, we get 
\begin{equation*}
 d(x_{k}, \Xi)\leq  \psi^{-1}\big(\gamma b \beta(\ell d(x_{k-1}\Xi)\big)=\tau(d(x_{k-1}, \Xi))
\end{equation*}
and thus verifies \eqref{0920}. It follows from Lemma~\ref{P10}(iii), Lemma~\ref{L1.1}, and Theorem~\ref{T5.3} (applied to $\ell_1=\ell_2=0$, $\ell_3=\ell$, and $s_k= d(x_k, \Xi)$) that \eqref{0920} holds for all large $k$, and hence $\{x_k\}$ converges to $\overline{x}\in \Xi$ as $k\to \infty$.

Using now \eqref{0920}, the increasing property of $\tau$, and the fact that $\overline{x}\in \Xi$,  we arrive at 
\begin{equation}\nonumber
d(x_{k}, \Xi)\leq \tau( d(x_{k-1}, \Xi))\leq \tau( \|x_{k-1}-\overline{x}\|)
\end{equation}
for all $k$ sufficiently large. Combining this with \eqref{256} and recalling that $\ell_1=\ell_2=0$, $\ell_3=\ell$, and $s_k= d(x_k, \Xi)$ in this case, we obtain the estimates
\begin{equation*}
\begin{aligned}
\|x_{k}-\overline{x}\|\leq c \, \Delta_k \leq  & \, c \, \ell \, d(x_{k}, \Xi)\sum_{n=0}^{\infty}\frac{\tau^n(d(x_{k}, \Xi))}{d(x_{k}, \Xi)}
\leq  c \,\ell \, \tau(\|x_{k-1}-\overline{x}\|)\sum_{n=0}^{\infty}\frac{\tau^n( d(x_{k}, \Xi))}{ d(x_{k}, \Xi)},
\end{aligned}
\end{equation*}
which being combined with Lemma~\ref{P10}(iii) and Lemma~\ref{L1.1} justifies \eqref{0910} and complete the proof.
\end{proof}

In the rest of this section, we present discussions on some remarkable features of our abstract convergence analysis and its comparison with known results in the literature.

\begin{remark}[{\bf discussions on abstract convergence analysis and its specifications}]\label{rem:abs-disc} $\,$ {\rm

$\bullet$ Note first that, for the three assumptions imposed in Theorem~\ref{P5.2}, the second one (surrogate sequence bounded in terms of the target set) is often tied up with and indeed guaranteed by the construction of many common algorithms, while the other two assumptions are related to some regularity of the function $f$ with respect to  the target set. In particular, in the special case where  $e_k=\|x_{k+1}-x_k\|$, Assumption~(ii) in Theorem~\ref{P5.2} is {\em automatically 
satisfied} for the proximal methods \cite{abs} with $\Xi=\Gamma$,  and for the cubic regularized Newton methods \cite{np06} with $\Xi=\Theta$, $p \in (0,q+1)$. 

$\bullet$ For the aforementioned {\em proximal method} from \cite{abs} to solve $\min_{x \in \mathbb{R}^m}f(x)$ with a fixed proximal parameter $\lambda>0$, where $f$ is a proper l.s.c.\ function, the automatic fulfillment of Assumption~(ii) allows us to deduce from Theorem~\ref{P5.2} that if $\partial f$ is $p$th-order exponent metrically subregular at $\overline{x}$ with respect to the first-order stationary set $\Gamma$ for $p \in (0,1)$, then the proximal method converges {\em at least superlinearly with the convergence rate $\frac{1}{p}$}. To the best of our knowledge, this appears to be new to the literature. Moreover, we see in Example~\ref{remark:4.2} below that the derived convergence rate can indeed be attained.  

$\bullet$ In Section~\ref{sec:alg}, we apply the conducted convergence analysis to the {\em higher-order regularized Newton method with momentum steps} to solve $\min_{x \in \mathbb{R}^m}f(x)$, where $f$ is a $C^2$-smooth function whose Hessian is H\"older continuous with exponent $q \in (0,1]$. As mentioned, Assumption~(ii) holds automatically with respect to the second-order stationary points $\Theta$ from \eqref{2stat}. Therefore, if $\partial f$ is $p$th-order exponent metrically subregular at $\overline{x}$ with respect to $\Theta$ for $p \in (0,q+1)$, then the sequence generated by the algorithm converges to a second-order stationary point {\em at least superlinearly with order {$\frac{q+1}{p}$}}. Even in the special case where $p=1$, this improves the current literature on high-order regularized Newton methods by obtaining fast convergence rates under the pointwise metric subregularity (instead of the uniform metric subregularity/error bound condition as in \eqref{eb}), and in the more general context with allowing momentum steps; see Remark~\ref{remark:5.1} for more details.}
\end{remark}

Now we provide a simple example showing that the established convergence rate {\em can be attained} for the proximal method of minimizing a smooth univariate function.

\begin{example}[{\bf tightness of the convergence rate}]\label{remark:4.2} {\rm Consider applying the proximal point method for the univariate function $f(t)=|t|^{\frac{3}{2}}$, $t\in\mathbb R$, with the iterations
\begin{equation*}
t_{k+1}={\rm argmin}_{t \in \mathbb{R}}\big\{f(t)+ \frac{\lambda}{2}(t-t_k)^2\big\},\quad t_0=1,
\end{equation*}
where $\lambda$ is a fixed positive parameter. Then we have
\begin{equation*}
\nabla f(t_{k+1})+\lambda (t_{k+1}-t_k)= \frac{3}{2}{\rm sign}(t_{k+1}) |t_{k+1}|^{\frac{1}{2}}+ \lambda (t_{k+1}-t_k)=0.
\end{equation*}
This shows that $t_k>0 \Rightarrow t_{k+1}>0$, and so $t_k>0$ for all $k$, which further implies that
$t_k=\frac{3}{2\lambda}(t_{k+1})^{\frac{1}{2}}+t_{k+1}$ yielding $t_k \ge \frac{3}{2\lambda}(t_{k+1})^{\frac{1}{2}}$.
Take $\Xi=\Gamma$, the first-order stationary set \eqref{1stat}, and get
\begin{equation*}
|t_{k+1}-t_k|=\frac{3}{2\lambda} |t_{k+1}|^{\frac{1}{2}} \le |t_k|=d(t_k,\Xi).
\end{equation*}
For $w_{k+1}=\nabla f(t_{k+1})=\frac{3}{2} t_{k+1}^{\frac{1}{2}}$, we have $|w_{k+1}|=\frac{3}{2} t_{k+1}^{\frac{1}{2}}= \lambda |t_{k+1}-t_k|$. Choosing now $\psi(t)=\frac{3}{2} t^{\frac{1}{2}}$ and $\beta(t)=t$ gives us $\tau(t)=\psi^{-1}(\lambda \beta( t))=O(t^2)$ and the equalities
$\psi(d(t_k,\Xi))=\psi(|t_k|)=\frac{3}{2} |t_k|^{\frac{1}{2}}=d(0,\nabla f(t_k))$. Direct verifications show that all the conditions in Theorem~\ref{P5.2} are satisfied, and thus it follows from the theorem that $t_k \rightarrow \overline{t}=0$ in a quadratic rate, e.g.,
$|t_{k+1}-\overline{t}|=t_{k+1}=O(t_k^2)=O(|{t_{k}-}\overline{t}|^2)$.

To determine the exact convergence rate, we solve $t_k=\frac{3}{2\lambda} \, t^{\frac{1}{2}}+t$ in terms of the variable $t\in\mathbb{R}_+$. Letting $u=\sqrt{t}$ provides $t_k=\frac{3}{2\lambda}u+u^2$, which implies that
$u= -\frac{3}{4 \lambda} + \sqrt{t_k + \frac{9}{16\lambda^2}}$. 
Passing finally to the limit as $k\to\infty$ and remembering that $t_k \rightarrow 0$ tell us that
\begin{equation*}
t_{k+1}= \left[-\frac{3}{4 \lambda} + \sqrt{t_k + \frac{9}{16\lambda^2}}\right]^2=\left[\frac{t_k}{\frac{3}{4 \lambda} + \sqrt{t_k + \frac{9}{16\lambda^2}}}\right]^2 =O(t_k^2),
\end{equation*}
which agrees with the exact quadratic convergence rate derived in Theorem~\ref{P5.2}.}
\end{example}\vspace*{-0.2in}

\section{Abstract Convergence Analysis under the KL Property}\label{sec:kl}\setcounter{equation}{0}\vspace*{-0.05in}

In this section, we continue our convergence analysis of abstract algorithms involving now the {\em Kurdyka-\L ojasiewicz} (KL) property from Definition~\ref{de1}. Here is the main result of this section.

\begin{theorem}[{\bf convergence under the KL properties}]\label{P5.1}
Consider a sequence $\{(x_k,e_k)\}$,  which satisfies assumptions  \eqref{H001}, \eqref{H2}, \eqref{H3} being such that the set $\mathcal{L}(f(x_{k_0}))$  is bounded for some $k_0\in\mathbb N$. Let $\overline{x}$ be an accumulation point of the iterated sequence, and let $f:\mathbb{R}^m\to\Bar{\mathbb{R}}$ be an l.s.c.\ function satisfying the KL property at $\overline{x}$, where $\eta>0$  and $\vartheta:[0, \eta)\to \mathbb{R}_+$ are taken from Definition~{\rm\ref{de1}}. Suppose also that there exists a pair $(\mu_1, \mu_2)\in (0, 1)\times \mathbb{R}_+$ for which \begin{equation}\label{regularity}
\limsup_{(s, t)\to (0,0)}\frac{\varphi^{-1}(\beta(s)t)}{\mu_1s+\mu_2 t}\leq 1, 
\end{equation}
where $\varphi(\cdot)$ and $\beta(\cdot)$ are increasing  functions taken from conditions \eqref{H001} and  \eqref{H2}. Then we have $\sum_{k=1}^{\infty}\|x_{k+1}-x_k\|<\infty$, and the sequence $\{x_k\}$ converges  to $\overline{x}$ as $k\to \infty$.
\end{theorem}
\begin{proof}
It follows from \eqref{regularity} that for any $\alpha\in \left(1, \frac{1}{\mu_1}\right)$ there exists $\delta>0$ such that
\begin{equation}\label{520}
\varphi^{-1}\big(\beta(s)t\big)\leq \alpha \mu_1 s+\alpha \mu_2 t\quad\mbox{ whenever }\;s, t\in (0, \delta).
\end{equation}
Let $\overline{\eta}:=\min\{\eta, \varepsilon\}$ with $\varepsilon,\eta$ given in Definition~\ref{de1}. It follows from the abstract convergence framework of Theorem~\ref{T5.3} with $s_k=0$ that the claimed assertions are verified if showing that there exist constants $\ell_1\in (0, 1), \ell_2\in [0,\infty)$, and $k_1 \in \mathbb{N}\setminus\{0\}$ such that for any $k \ge k_1$ with $x_k\in B_{\mathbb{R}^m}(\overline{x}, \overline{\eta})$ we have
\begin{equation}\label{5.287}
 e_k\leq  \ell_1e_{k-1}+\ell_2\Lambda_{k,k+1}
\end{equation}
with the quantities $\Lambda_{k, k+1}$ defined by
\begin{equation*}
\Lambda_{k, k+1}:= \vartheta\big(f(x_k)-f(\overline{x})\big)-\vartheta\big(f(x_{k+1})-f(\overline{x})\big).
\end{equation*}
To see this, recall that $\varphi$ is an admissible function and  then deduce from \eqref{H001} and Lemma \ref{P10} that 
\begin{equation*}
f(x_{k+1})\leq f(x_k)\;\mbox{ for all }\;k\in \mathbb{N},\quad f(\overline{x})=\inf_{k\in\mathbb{N}}f(x_k)=\lim_{k\to \infty}f(x_k),
\end{equation*}
and $\lim_{k\to \infty}e_k=0$. It follows from the continuity of $\vartheta$ that there exists $k_1\geq 1$ for which  
\begin{equation}\label{291}
 f(\overline{x})\leq f(x_k)<f(\overline{x})+\overline{\eta},\quad e_{k-1}<\delta,\;\mbox{ and }
\end{equation}
\begin{equation}\label{292}
 \frac{b}{a}[\vartheta(f(x_k)-f(\overline{x}))-\vartheta(f(x_{k+1})-f(\overline{x}))]<\delta\;\mbox{ whenever }\; k\geq k_1.
\end{equation}
Take any $k\geq k_1$ with $x_k\in B_{\mathbb{R}^m}(\overline{x}, \overline{\eta})$ and observe that if $e_k=0$, then \eqref{5.287} holds trivially. Assuming that $e_k\neq 0$ and then using \eqref{H001} and \eqref{291}, we get $f(\overline{x})\leq f(x_{k+1})<f(x_k)$, which being combined with the KL property and condition \eqref{H2} shows that $w_k\neq 0$ and $e_{k-1}\neq 0$. Since $w_k\in \partial f(x_{k})$, the aforementioned conditions readily yield the estimates
\begin{equation}\label{5237}
\vartheta'(f(x_k)-f(\overline{x}))\geq \frac{1}{\|w_k\|}\geq \frac{1}{b\beta(e_{k-1})}.
\end{equation}
The concavity of $\vartheta$ in the KL property and assumption \eqref{H001} imply that 
\begin{equation}\label{5238}
\begin{aligned}
\vartheta\big(f(x_k)-f(\overline{x})\big)-\vartheta(f(x_{k+1})-f(\overline{x}))&\ge\vartheta'\big(f(x_k)-f(\overline{x})\big)\big(f(x_k)-f(x_{k+1})\big)\\
&\ge a \vartheta'\big(f(x_k)-f(\overline{x})
\big)\varphi( e_k).
\end{aligned}
\end{equation}
Substituting \eqref{5237} into \eqref{5238} guarantees that
\begin{equation*}\label{239}
\frac{b}{a}\big[\vartheta\big(f(x_k)-f(\overline{x})\big)-\vartheta\big(f(x_{k+1})-f(\overline{x}\big)\big]\geq \frac{\varphi(e_k)}{\beta(e_{k-1})},
\end{equation*}
which establishes in turn that
\begin{equation*}
\begin{aligned}
e_k\leq \varphi^{-1}\Big(\frac{b}{a}\beta(e_{k-1})\big[\vartheta\big(f(x_k)-f(\overline{x})\big)-\vartheta\big(f(x_{k+1})-f(\overline{x})\big)\big]\Big).
\end{aligned}
\end{equation*}
Combining the latter with \eqref{520} and \eqref{292}, we get
\begin{equation*}
\begin{aligned}
e_k&\leq\alpha\mu_1e_{k-1}+\frac{\alpha \mu_2b}{a}\Big(\vartheta\big(f(x_k)-f(\overline{x})\big)-\vartheta\big(f(x_{k+1})-f(\overline{x})\big)\Big),
\end{aligned}
\end{equation*}
which shows that \eqref{5.287} holds with $\ell_1:=\alpha\mu_1<1$ and $\ell_2:= \frac{\alpha \mu_2b}{a}$ and thus completes the proof.
\end{proof}

Now we employ Theorem~\ref{P5.1} and the results of the previous section to find explicit conditions ensuring {\em superlinear convergence} of an abstract algorithm under the KL property. Talking $\vartheta\colon[0,\eta]\to\mathbb R_+$ from Definition~\ref{de1} and the admissible functions $\ph(\cdot)$ and $\beta(\cdot)$ from Theorem~\ref{P5.1}, define $\rho:\mathbb{R}_+\to \mathbb{R}_+$ by
\begin{equation*}
\rho(t):=\left\{ \begin{array}{cc} \big((\vartheta^{-1})'\big)^{-1}(b\beta(t)) & \mbox{ if }\;t>0,\\
0 & \mbox{ if }\;t=0  
\end{array} \right.
\end{equation*}
and consider the function  $\xi:\mathbb{R}_+\to \mathbb{R}_+$ given by $\xi(t):=\varphi^{-1}\left(\frac{\vartheta^{-1}(\rho(t))}{a}\right)$ for all $t\in \mathbb{R}_+$.

\begin{theorem}[{\bf superlinear convergence rate under the KL property}]\label{kl-super} Let $\{(x_k, e_k)\}$, $\mathcal{L}(f(x_{k_0}))$, $\overline{x}$, $f$, $\vartheta$, $\varphi$, and $\beta$ be as in the general setting of Theorem~{\rm\ref{P5.1}} under the assumptions \eqref{H001},  \eqref{H2}, and \eqref{H3}. Take the function $\xi:\mathbb{R}_+\to \mathbb{R}_+$ as defined above and assume in addition that it is  increasing on $\mathbb{R}_+$ with $\lim_{t\to 0^+}\frac{\xi(t)}{t}=0$. The following assertions hold:
\begin{itemize}
\item[\bf(i)] If there exists $r \ge 1$ such that  $e_k\leq r\|x_{k+1}-x_k\|$ for all large $k\in \mathbb{N}$, then the sequence $\{x_k\}$ converges to $\overline{x}$ as $k\to \infty$ with the convergence rate
\begin{equation}\label{1.36}
\limsup_{k\to \infty}\frac{\|x_{k}-\overline{x}\|}{\widetilde{\xi}(\|x_{k-1}-\overline{x}\|)}<\infty,\;\mbox{ where }\;\widetilde{\xi}(t):=\xi(2r t).
\end{equation}
\item[\bf(ii)] If there exist  $\ell>0$ and a closed set $\Xi$ satisfying $\Omega\subset \Xi\subset \Gamma$ and such that
\begin{equation}\label{443}
e_k\leq \ell \, d(x_k, \Xi)\;\mbox{ for all large }\;k\in\mathbb N,
\end{equation}
then the sequence $\{x_k\}$ converges to $\overline{x}\in \Xi$ as $k\to \infty$ with the convergence rate
\begin{equation}\label{444}
\limsup_{k\to \infty}\frac{\|x_{k}-\overline{x}\|}{\widehat{\xi}(\|x_{k-1}-\overline{x}\|)}<\infty,\;\mbox{ where }\;\widehat{\xi}(t):=\xi(\ell \, t).
\end{equation}
\end{itemize}
\end{theorem}
\begin{proof} Suppose without loss of generality that $e_k\neq 0$ for all $k\in\mathbb N$. It follows from Theorem~\ref{T5.3} that \eqref{5237} holds when $k$ is sufficiently large and that there exists $k_1$ such that
\begin{equation*}
(\vartheta^{-1})'\big(\vartheta((f(x_k)-f(\overline{x})))\big)=\frac{1}{\vartheta'\big(f(x_k)-f(\overline{x})\big)}\leq  b\beta(e_{k-1})\quad\mbox{as }\; k\geq k_1.
\end{equation*}
The latter readily implies that
\begin{equation}\label{301}
\vartheta\big((f(x_k)-f(\overline{x}))\big)\leq   \big((\vartheta^{-1})'\big)^{-1}\big(b\beta(e_{k-1})\big) =\rho(e_{k-1}) \quad\mbox{as }\; k\geq k_1.
\end{equation}
Furthermore, condition \eqref{H001} and the first part of \eqref{291} tell us that
\begin{equation*}
a\varphi(e_k)\leq f(x_k)-f(x_{k+1})\leq f(x_k)-f(\overline{x})\;\mbox{ for all }\;k\in\mathbb{N},
\end{equation*}
which together with \eqref{301} yields $e_k\le \xi(e_{k-1})$.
Combining this with the limiting conditions $\lim_{k\to \infty}e_{k}=0$ and $\lim_{t\to 0^+}\frac{\xi(t)}{t}=0$ allows us to find a sufficiently large number $k_2\in\mathbb N$ such that
\begin{equation*}
e_k\leq \xi(e_{k-1})<\frac{1}{2} e_{k-1}\quad\mbox{for all }\;k\geq k_2.
\end{equation*}
Invoking again condition \eqref{H001} ensures that for any $k\geq k_2$ we have the estimates
\begin{equation}\label{1.74}
\begin{aligned}
\|\overline{x}-x_{k}\|=\lim_{j\to \infty}\|x_{k+j}-x_{k}\|\leq &\sum_{i=k}^{\infty}\|x_{i+1}-x_i\|\leq c\sum_{i=k}^{\infty}e_i
\leq  2c \, e_k\leq 2c \, \xi(e_{k-1}).
\end{aligned}
\end{equation}

Now we are in a position to justify both assertions of the theorem. Starting with (i), note that
\begin{equation*}
\xi(e_{k-1})\leq \xi(r\|x_k-x_{k-1}\|)\leq \xi(r\|x_k-\overline{x}\|+r\|x_{k-1}-\overline{x}\|)\leq \xi(2r\|x_k-\overline{x}\|)+\xi(2r\|x_{k-1}-\overline{x}\|)
\end{equation*}
for large $k$, where the last inequality follows from the increasing property of $\xi$ ensuring that 
\begin{equation*}
\xi(\alpha+\gamma) \le \max\big\{\xi(2 \alpha),\xi(2\gamma)\big\} \le \xi(2 \alpha)+\xi(2\gamma)
\end{equation*}
for any $\alpha,\gamma \ge 0$. Observing that $\lim_{t\to 0^+}\frac{\xi(2rt)}{t}=0$, it readily follows from $x_k \to \overline{x}$  and \eqref{1.74} that
\begin{equation*}
\begin{aligned}
\limsup_{k\to \infty}\frac{\|x_{k}-\overline{x}\|}{\xi(2r\|x_{k-1}-\overline{x}\|)}= &\limsup_{k\to \infty}\frac{\|x_{k}-\overline{x}\|}{\xi(2r\|x_{k-1}-\overline{x}\|)+\xi(2r\|x_{k}-\overline{x}\|)}\\
\le&\limsup_{k\to \infty}\frac{\|x_{k}-\overline{x}\|}{\xi(e_{k-1})}<\infty,
\end{aligned}
\end{equation*}
which brings us to \eqref{1.36} and thus justifies assertion (i).

To verify finally assertion (ii), observe that the increasing property of $\xi$ and the inclusion $\overline{x}\in \Xi$ allow us to deduce from \eqref{443} the estimates
\begin{equation*}
\xi(e_{k-1})\le\xi\big(\ell d(x_{k-1}, \Xi)\big)\leq \xi\big(\ell\|x_{k-1}-\overline{x}\|\big)\quad\mbox{ for all }\;k\ge k_2,
\end{equation*}
which together with \eqref{1.74} give us \eqref{444} and thus completes the proof of the theorem.
\end{proof}

We conclude this section with some specifications of the convergence results obtained under the KL property and comparisons with known results in the literature. 

\begin{remark}[{\bf further discussions and comparisons for KL convergence analysis}]\label{remark:4.4} $\,$ {\rm 

$\bullet$ Observe first that assumption \eqref{regularity} in Theorem~\ref{P5.1} holds {\em automatically} for many common algorithms. For example, in the standard proximal methods \cite{abs} we have $\varphi(t)=t^2$ and $\beta(t)=t$, and thus by setting $(\mu_1, \mu_2)=(1/2, 1/2)$ in Theorem~\ref{P5.1}, condition \eqref{regularity} holds by the classical AM–GM inequality.

$\bullet$ In the case where $\vartheta(t)=c t^{1-\theta}$ with the {\em KL exponent} $\theta\in[0,1)$, we can obtain the explicit {\em sublinear, linear}, and {\em superlinear} convergence rates. In fact, let $\beta, \varphi,\vartheta:\mathbb{R}_+\to \mathbb{R}_+$ be such that
$\beta(t)=t^{\alpha},\; \varphi(t)=t^{1+\alpha},\;\mbox{ and }\vartheta(t)=ct^{1-\theta}$, where $c,\alpha >0$ and $\theta \in [0,1)$.
It is easy to see that the assumption \eqref{regularity} in Theorem~\ref{P5.1} holds with $(\mu_1, \mu_2)=(\frac{\alpha}{\alpha+1}, \frac{1}{\alpha+1})$ due to the AM–GM inequality. In this case, direct verifications show that
\begin{equation*}
\rho(t):=\big((\vartheta^{-1})'\big)^{-1}(b\beta(t)) =(1-\theta)^\frac{1-\theta}{\theta}c^{\frac{1}{\theta}}b^{\frac{1-\theta}{\theta}}t^{\frac{\alpha(1-\theta)}{\theta}}=:\kappa t^{\frac{\alpha(1-\theta)}{\theta}},
\end{equation*}
\begin{equation*}
\xi(t):=\varphi^{-1}\left(\frac{\vartheta^{-1}(\rho(t))}{a}\right)=a^{-\frac{1}{1+\alpha}}\left(\frac{\kappa}{c}\right)^{\frac{1}{(1+\alpha)(1-\theta)}} t^\frac{\alpha}{(1+\alpha)\theta}=:\mu t^\frac{\alpha}{(1+\alpha)\theta}.
\end{equation*}
Suppose that condition \eqref{443} holds, which is true for various common numerical methods such as proximal-type methods and high-order regularized methods. If $\theta<\frac{\alpha}{1+\alpha}$, then Theorem~\ref{kl-super} implies that the sequence $\{x_k\}$ converges to $\overline{x}$  {\em superlinearly} with the convergence rate $\frac{\alpha}{(1+\alpha)\theta}$. It will be seen in Section~\ref{sec:alg} that if we further assume that $f$ is a ${\cal C}^2$-smooth function and satisfies the {\em strict saddle point property}, then Theorem~\ref{P5.2} and Corollary~\ref{coro3.1} allow us to improve the superlinear convergence rate to $\frac{(1-\theta)\alpha}{\theta}$. 

$\bullet$ Following the arguments from \cite[Theorem~2]{ab09}, we can also obtain convergence rates in the remaining cases for $\theta$. Indeed, when $ \theta=\frac{\alpha}{1+\alpha}$ there exist $\varrho\in (0,1)$ and $\zeta\in (0,\infty)$ such that for large $k$ we get
\begin{equation}\nonumber
\|x_k-\overline{x}\| \le \Delta_k:=\sum_{i=k}^{\infty}\|x_{i+1}-x_i\|\leq \zeta\varrho^{k}.
\end{equation}
If finally $\theta> \frac{\alpha}{1+\alpha}$, then there exists $\mu>0$ ensuring the estimates 
\begin{equation*}
\|x_k-\overline{x}\|\le\Delta_k\le\mu k^{-\frac{\alpha(1-\alpha)}{(1+\alpha)\theta-\alpha}}\;\mbox{ for all large }\;k\in\mathbb N.
\end{equation*}
The proof of this fact is rather standard, and so we omit details here.

$\bullet$ Let us highlight that the superlinear convergence rate derived under the KL assumptions need {\em not be tight} in general, and it can be worse than the one derived under the generalized metric subregularity. For instance, consider the situation in Example~\ref{remark:4.2}, where the proximal point method is applied to $f(t)=|t|^{\frac{3}{2}}$. As discussed in that example, the sequence $\{t_k\}$ generated by the algorithm converges with a quadratic rate, and the convergence rate agrees with the rate derived under the generalized metric subregularity. Direct verification shows that the function $f(t)=|t|^{\frac{3}{2}}$ satisfies the KL property with exponent $\theta=\frac{1}{3}$, and that we have $\alpha=1$ for the proximal point method. Thus Theorem~\ref{kl-super} only implies that $\{t_k\}$ converges {\em superlinearly} with rate $\frac{3}{2}$ while {\em not quadratically} as under generalized metric subregularity.}
\end{remark}
\vspace*{-0.2in}

\section{Generalized Metric Subregularity for Second-Order Conditions}\label{sec:2nd}\setcounter{equation}{0}\vspace*{-0.05in}

In this section, we continue our study of generalized metric subregularity while concentrating now on deriving verifiable conditions for the fulfillment of this notion with respect to the set $\Xi=\Theta$ of {\em second-order stationary points} defined in \eqref{2stat}, where the function $f$ in question is ${\cal C}^2$-smooth. Such a {\em second-order generalized metric subregularity} plays a crucial role in the justification of fast convergence of our new high-order regularized Newton method in Section~\ref{sec:alg}. We'll see below that this notion holds in broad settings of variational analysis and optimization and is satisfied in important practical models.

Note that if $f$ is convex, then there is no difference between $\Theta$ and the set $\Gamma$ of first-order stationary points \eqref{1stat}. We focus in what follows on {\em nonconvex} settings for generalized metric subregularity, where $\Theta\ne\Gamma$ as in the following simple while rather instructive example. 

\begin{example}[\bf generalized metric subregularity for nonconvex functions]\label{2noncov} {\rm Consider the function $f:\mathbb{R}^m \rightarrow \mathbb{R}$ given by $f(x)=(\|x\|^2-r)^{2p}$ with $r>0$ and $p \ge 1$. Direct computations show that $\nabla f(x)= {4}p (\|x\|^2-r)^{2p-1} x$ and thus 
$$\nabla^2 f(x)= {8}p(2p-1) (\|x\|^2-r)^{2p-2}xx^T +  {4}p (\|x\|^2-r)^{2p-1} I_m,$$
where $I_m$ stands for the $m\times m$ identity matrix. Since 
$\nabla^2f(0)$ is negative-definite, $f$ is nonconvex. Moreover,
$\Gamma=\{x\;|\;\|x\|=\sqrt{r}\} \cup \{0\}$ and $\Theta=\{x\;|\;\|x\|=\sqrt{r}\}$. Therefore, for any $\overline{x} \in \Theta$ we have
\begin{eqnarray*}
&&d(x,\Theta)=\big|\|x\|-\sqrt{r}\big| = \frac{\big|\|x\|^2-r\big|}{\big|\|x\|+\sqrt{r}\big|} \le\Big(\frac{3}{2} \sqrt{r}\Big)^{-1} \big|\|x\|^2-r\big| 
\\
&&=\Big(\frac{3}{2} \sqrt{r}\Big)^{-1} \frac{[{4}p \big|\|x\|^2-r\big|^{2p-1} \|x\|]^{\frac{1}{2p-1}}}{({4}p\|x\|)^{\frac{1}{2p-1}}}
\le(\frac{3}{2} \sqrt{r} )^{-1} \frac{\|\nabla f(x)\|^{\frac{1}{2p-1}}}{({2}p\sqrt{r})^{\frac{1}{2p-1}}}\;\mbox{ whenever }\;x \in B_{\mathbb{R}^m}(\overline{x},\sqrt{r}/2),
\end{eqnarray*}
and thus $\nabla f$ satisfies the generalized metric subregularity with respect to $(\psi,\Xi)$ for $\psi(t)=t^{2p-1}$ and $\Xi=\Theta$.}
\end{example}

Let us now introduce the new property that is playing an important role below. Given a proper l.s.c.\ function $f\colon\mathbb{R}^m\to\Bar{\mathbb R}$, we say that $f$  has the {\em weak separation property} (WSP) at $\overline{x} \in \Gamma$ if there is $\varsigma>0$ with
\begin{equation}\label{se010}
f(y) \le f(\overline{x})\quad\mbox{for all }\;y\in \Gamma\cap B_{\mathbb{R}^m}(\overline{x},\varsigma),
\end{equation}
where $\Gamma$ is taken from \eqref{1stat}. The constant $\varsigma$ is called the {\em modulus of weak separation}. If the inequality in \eqref{se010} is replaced by equality, such property is frequently used in the study of error bounds under the name of ``separation property" (e.g., in \cite{lp,lts}), which inspires our terminology, While the separation property obviously yields WSP, the converse fails. For example,  consider $f(x):= - e^{\frac{-1}{x^2}} [\cos(\frac{1}{x})]^{2p}$ if $x \neq 0$ and $f(x):=0$ if $x =0$, where $p \in \mathbb{N}$. Take $\overline{x}=0$ and deduce from the definition that $f$ satisfies WSP at $\overline{x}$. On the other hand, $f(\frac{1}{2k \pi+\frac{\pi}{2}})=f(\frac{1}{2k \pi+\frac{3\pi}{2}})=0$ for  all $k \in \mathbb{N}$, and hence there exists $a_k \in (\frac{1}{2k \pi+\frac{3\pi}{2}},\frac{1}{2k \pi+\frac{\pi}{2}})$ such that $f'(a_k)=0$ and $f(a_k)<0=f(\overline{x})$, which tells us that the separation property fails.\vspace*{0.05in} 

The next proposition shows that a large class of functions satisfies the weak separation property.

\begin{proposition}[\bf sufficient conditions for weak separation]\label{le3.2} Let $f\colon\mathbb{R}^m\to\Bar{\mathbb R}$ be a proper l.s.c.\ function, and let $\overline{x} \in \Gamma$. Suppose that either one of the following conditions holds:
\begin{itemize}
\item[{\bf(i)}]  $f$ is continuous at $\overline{x}$ with respect to ${\rm dom}(\partial f)$, i.e., $\lim\limits_{\substack{z \in {\rm dom}(\partial f)}{z \rightarrow \overline{x}}}f(z)=f(\overline{x})$, and it satisfies the KL property at $\overline{x} \in \Gamma$.

\item[{\bf(ii)}] $f=\varphi \circ g$, where $\varphi:\mathbb{R}^n \rightarrow \mathbb{R}$ is a convex function, and where $g:\mathbb{R}^m \rightarrow \mathbb{R}^n$ is a smooth mapping such that the Jacobian matrix $\nabla g(\overline{x})$ is surjective.   
\end{itemize} Then $f$ enjoys the weak separation property at $\overline{x}$.
\end{proposition}
\begin{proof} To verify (i), suppose on the contrary that WSP fails at $\ox$. Then there exists a sequence $y_k \rightarrow \overline{x}$ with $y_k \in \Gamma$ such that $f(y_k)>f(\overline{x})$ for all $k\in\mathbb N$. It follows from the continuity of $f$ at $\overline{x}$ with respect to ${\rm dom}(\partial f)$ that
$y_k \in \{x\;|\;f(\overline{x})<f(x)< f(\overline{x})+\eta\}$ for all large $k$, where $\eta$ is taken from the assumed KL property at $\overline{x}$. The latter condition implies that
\begin{equation*}
\vartheta'(f(y_k)-f(\overline{x}))d(0, \partial f(y_k))\geq 1
\end{equation*}
while contradicting $y_k \in \Gamma$ and hence justifying the claimed WSP. Assertion (ii) follows directly from Lemma~\ref{LE2.1}, which therefore completes the proof of the proposition.
\end{proof}

Another concept used below to study the second-order generalized metric subregularity applies to ${\cal C}^2$-smooth functions. Given such a function, recall that $\overline{x}$ is its {\em strict saddle point} if $\lambda_{\min}(\nabla^2 f(\overline{x}))<0$. We say that $f$ satisfies the {\em strict saddle point property} at $\overline{x} \in \Gamma$ if $\overline{x}$ is either a local minimizer for $f$, or a strict saddle point for $f$. If $f$ satisfies the strict saddle point property at any first-order stationary point, then we simply say that $f$ satisfies the strict saddle point property. Furthermore, if for any first-order stationary point $\overline{x}$ of $f$, it is either a global minimizer for $f$, or a strict saddle point for $f$, then we say that $f$ satisfies the {\em strong strict saddle point property}. It is known from Morse theory that strict saddle point property holds {\em generically} for ${\cal C}^2$-smooth functions. Moreover, it has been recognized in the literature that various classes of nonconvex functions arising in matrix completion, phase retrieval, neural networks, etc. enjoy the strict saddle point property; see some instructive examples below.\vspace*{0.03in}

Before establishing the main result of this section, we present one more proposition of its own interest telling us that, around a {\em local minimizer}, the distance to $\Gamma$, the set of first-order stationary points, agrees with the distance to $\Theta$, the set of second-order stationary points, {\em under WSP}. Let us emphasize that in general $\Gamma$ and $\Theta$ are {\em not the same} while the distances to $\Gamma$ and $\Theta$ are under the imposed assumptions.

\begin{proposition}[\bf distance to stationary points]\label{L34}
 Given a ${\cal C}^2$-smooth function $f\colon\mathbb{R}^m\to\mathbb R$ and $\overline{x}\in \mathop{\arg\min}_{x\in B_{\mathbb{R}^m}[\overline{x}, \alpha]}f(x)$ for some $\alpha>0$, suppose that WSP holds at $\overline{x}$ with modulus $\varsigma>0$. Denoting $\gamma:=\min\{\alpha, \varsigma\}$, we have the equality
\begin{equation}\label{se011}
d(x, \Theta)=d(x, \Gamma)=d\big(x, \mathop{\arg\min}_{x\in B_{\mathbb{R}^m}[\overline{x}, \gamma]}f(x)\big)=d\big(x,[f \le f(\overline{x})]\big)\;\mbox{ for all }\;x\in B_{\mathbb{R}^m}\big(\overline{x},\gamma/4\big).
\end{equation}
\end{proposition}
\begin{proof} It follows from \eqref{se010}, the fact that  $\overline{x}\in \mathop{\arg\min}_{x\in B_{\mathbb{R}^m}[\overline{x}, \alpha]}f(x)$, and the choice of $\gamma$ that
\begin{equation*}
\Gamma\cap  B_{\mathbb{R}^m}(\overline{x}, \gamma)\subset \mathop{\arg\min}_{x\in B_{\mathbb{R}^m}[\overline{x}, \gamma]}f(x)\subset [f \le f(\overline{x})]. 
\end{equation*}
This implies that for any $x\in B_{\mathbb{R}^m}(\overline{x},\gamma/2)$ we get  
\begin{equation}\label{l73}
\begin{aligned}
 &d(x, [f \le f(\overline{x})])\leq  d\big(x, \mathop{\arg\min}_{x\in B_{\mathbb{R}^m}[\overline{x}, \gamma]}f(x)\big)\leq\ d(x,  \Gamma\cap  B_{\mathbb{R}^m}(\overline{x}, \gamma))\\
&= \ \min\{d(x,  \Gamma\cap  B_{\mathbb{R}^m}(\overline{x}, \gamma)), d(x,  \Gamma\setminus  B_{\mathbb{R}^m}(\overline{x}, \gamma))\}
=\ d(x, \Gamma)\leq d(x, \Theta),
\end{aligned}
\end{equation}
where the first equality is due to $d(x,  \Gamma\cap  B_{\mathbb{R}^m}(\overline{x}, \gamma))<\gamma/2$ (by $\overline{x}\in \Gamma$) and $d(x,  \Gamma\setminus  B_{\mathbb{R}^m}(\overline{x}, \gamma)) \ge\gamma/2$. Recalling that $\overline{x}\in \mathop{\arg\min}_{x\in B_{\mathbb{R}^m}[\overline{x}, \alpha]}f(x)$ leads us to
\begin{equation*}
[f \le f(\overline{x})]\cap B_{\mathbb{R}^m}\left(\overline{x},\gamma/2\right)= \mathop{\arg\min}_{x\in B_{\mathbb{R}^m}\left(\overline{x},\gamma/2\right)}f(x)
\end{equation*} 
and to the distance estimate $\displaystyle d\big(x, \mathop{\arg\min}_{x\in B_{\mathbb{R}^m}\left(\overline{x},\gamma/2\right)}f(x)\big)\leq \|x-\overline{x}\|<\gamma/4$ for all $x\in B_{\mathbb{R}^m}\left(\overline{x}, \gamma/4\right)$. Therefore, taking $x$ from the above neighborhood of $\ox$ brings us to the equality
\begin{equation}\label{l74}
\begin{aligned}
d\big(x, [f \le f(\overline{x})]\big)= &\min\big\{d\big(x, \mathop{\arg\min}_{x\in B_{\mathbb{R}^m}\left(\overline{x},\gamma/2\right)}f(x)\big), d \big(x, [f \le f(\overline{x})]\setminus B_{\mathbb{R}^m}\left(\overline{x},\gamma/2\right)\big)\big\}\\=& \ d\big(x, \mathop{\arg\min}_{x\in B_{\mathbb{R}^m}\left(\overline{x},\gamma/2\right)}f(x)\big).
\end{aligned}
\end{equation}
Observing finally that $\mathop{\arg\min}_{x\in B_{\mathbb{R}^m}\left(\overline{x},\gamma/2\right)}f(x)\subset \Theta$, we deduce from \eqref{l73} and \eqref{l74} that
\begin{equation*}
\begin{aligned}
d(x, [f \le f(\overline{x})])\leq d\big(x, \mathop{\arg\min}_{x\in B_{\mathbb{R}^m}[\overline{x}, \gamma]}f(x)\big)\leq  d(x, \Gamma)\leq d(x, \Theta)
\leq d\big(x, \mathop{\arg\min}_{x\in B_{\mathbb{R}^m}\left(\overline{x}, \gamma/2\right)}f(x)\big)=d(x, [f \le f(\overline{x})])
\end{aligned}
\end{equation*}
for any $x\in B_{\mathbb{R}^m}\left(\overline{x}, \gamma/4\right)$, which justifies \eqref{se011} and thus completes the proof of the proposition.
\end{proof}

Now we are ready to derive sufficient conditions for generalized metric subregularity broadly used below.

\begin{theorem}[{\bf sufficient conditions for generalized metric subregularity}]\label{P3.3}  Let $f:\mathbb{R}^m\to \mathbb{R}$ be a ${\cal C}^2$-smooth function, let $\psi:\mathbb{R}_+\to \mathbb{R}_+$ be an admissible function, and let $\overline{x} \in \Theta$.  Suppose that both WSP and strict saddle point property hold for $f$ at $\overline{x}$. Then $\nabla f$ has the generalized metric subregularity property with respect to $(\psi, \Theta)$  at $\overline{x}$ if and only if $\nabla f$ has this property at $\overline{x}$ with respect to  $(\psi,\Gamma)$. In particular, if $f$ is a subanalytic ${\cal C}^2$-smooth function satisfying the strict saddle point property, then $\nabla f$ enjoys the generalized metric subregularity property with respect to $(\psi, \Theta)$  at $\overline{x} \in \Theta$, where $\psi(t)=t^{\gamma}$ for some $\gamma>0$.
\end{theorem}
\begin{proof}
Note first that by $\overline{x} \in \Theta$ we have $\nabla f(\overline{x})=0$ and $\nabla^2 f(\overline{x})\succeq 0$. Since $f$ satisfies the strict saddle point property, $\overline{x}$ must be a local minimizer for $f$. Therefore, the first conclusion of the theorem follows immediately from Proposition~\ref{L34}.

Suppose now that $f$ is a subanalytic ${\cal C}^2$-smooth function satisfying the strict saddle point property. As mentioned earlier, subanalytic functions satisfy the KL property \cite{BM}, and hence we deduce from Proposition~\ref{le3.2} that the weak separation property holds at $\overline{x}$. By \cite[Proposition 2.13(ii)]{bdl07}, the subanalytic property of $f$ ensures that the functions $h_1(x):= d(0,\nabla f(x))$ and $h_2(x):=d(x,\Gamma)$ are subanalytic. Observing that {$h_1^{-1}(0) \subset  h_2^{-1}(0)$}, we deduce from the Lojasiewicz factorization lemma for subanalytic functions in \cite[Theorem 6.4]{BM} that for each compact set $K$ there exist $c>0$ and $\gamma_0 \in (0,1]$ such that
$|h_2(x)| \le c|h_1(x)|^{\gamma_0}$ whenever $x\in K$. This verifies the second conclusion of the theorem and thus completes the proof.
\end{proof}

Next we provide some sufficient conditions for the {\em exponent metric subregularity} with explicit exponents for second-order stationary sets.

\begin{corollary}[\bf exponent metric subregularity for second-order stationary points]\label{coro3.1} Let $f\colon\mathbb{R}^m\to\mathbb R$ be a ${\cal C}^2$-smooth function, and let $\overline{x} \in \Theta$ for the second-order stationary set \eqref{2stat}. Suppose that either one of the following conditions holds:
\begin{itemize}
\item[\bf(i)] $f$ satisfies the strict saddle point property and the KL property at $\overline{x}$ with the KL exponent $\theta$.

\item[\bf(ii)] $f=\varphi \circ g$, where $\varphi:\mathbb{R}^n \rightarrow \mathbb{R}$ is a convex function satisfying the KL property at $g(\overline{x})$ with the KL exponent $\theta$, and where $g:\mathbb{R}^m \rightarrow \mathbb{R}^n$ is a smooth mapping with the surjective derivative $\nabla g(\overline{x})$.
\end{itemize}
Then $\nabla f$ enjoys the generalized metric subregularity property  at $\overline{x}$ with respect to $(\psi,\Theta)$, where $\psi(t)=t^{\frac{\theta}{1-\theta}}$.
\end{corollary}
\begin{proof} For (i), we get by the KL property of $f$ at $\overline{x}$ with the exponent $\theta$ that there are $c,\eta,\epsilon>0$ such that
\begin{equation}\label{kl0}
d(0, \nabla f(x))\geq c^{-1}\big(f(x)-f(\overline{x})\big)^{\theta}
\end{equation}
whenever $x\in B_{\mathbb{R}^m}(\overline{x}, \epsilon)\cap [f(\overline{x})<f<f(\overline{x})+\eta]$. Using the continuity of $f$ and shrinking $\epsilon$ if needed tell us that \eqref{kl0} holds for all $x\in B_{\mathbb{R}^m}(\overline{x}, \epsilon)\cap [f>f(\overline{x})]$. This together with \cite[Lemma 6.1]{LM} implies that there exist $\alpha>0$ and $\epsilon_0 \in (0,\epsilon)$ such that
\begin{equation*}
d(x, [f \le f(\overline{x})]) \le \alpha \, \big(f(x)-f(\overline{x})\big)^{1-\theta}\;\mbox{ for all }\;x \in B_{\mathbb{R}^m}(\overline{x}, \epsilon_0).
\end{equation*}
Shrinking $\epsilon_0$ again if necessary, we deduce from Propositions~\ref{le3.2} and  \ref{L34} that
\begin{equation*}
d(x,\Theta) \le \alpha \, \big(f(x)-f(\overline{x})\big)^{1-\theta}\;\mbox{ whenever }\;x \in B_{\mathbb{R}^m}(\overline{x}, \epsilon_0).
\end{equation*}
Thus there exists $M>0$ with $d(x,\Theta) \le M \, d\big(0, \nabla f(x)\big)^{\frac{1-\theta}{\theta}}$ for all $x \in B_{\mathbb{R}^m}(\overline{x}, \epsilon_0)$, and so (i) holds.

To justify (ii) under the imposed assumptions, it follows from
\cite[Theorem~3.2]{lp} that $f$ satisfies the KL property at $\overline{x}$ with the KL exponent $\theta$. Since $\overline{x} \in \Theta \subset \Gamma$, we get from Lemma~\ref{LE2.1} that $f$ satisfies the strict saddle point property at $\overline{x}$, and therefore the claimed conclusion follows from part (i).
\end{proof}

In the rest of this section, we discuss three practical models that often appear in, e.g., {\em machine learning}, {\em image processing}, and {\em statistics} for which the obtained sufficient conditions allow us to verify the second-order generalized metric subregularity.

\begin{example}[\bf over-parameterization of compressed sensing models]\label{ex:sensing}  {\rm Consider the {\em $\ell_1$-regularization problem}, known also as the {\em Lasso problem}: 
\begin{equation} \label{eq:L1}
\min_{x \in \mathbb{R}^m} \|Ax-b\|^2+ \nu \|x\|_1,
\end{equation}
where $A \in \mathbb{R}^{n\times m}$, $b \in \mathbb{R}^n$, $\nu>0$, and $\|\cdot\|_1$ is the usual $\ell_1$ norm. A recent interesting way to solve this problem \cite{tk24,poon,Poon} is to transform it into an {\em equivalent smooth problem } 
\begin{equation} \label{eq:OP}
\min_{(u,v) \in \mathbb{R}^m \times \mathbb{R}^m} f_{OP}(u,v):=\|A (u \circ v)-b\|^2+ \nu (\|u\|^2+\|v\|^2),
\end{equation}
where $u \circ v$ is the Hadamard (entrywise) product between the vector $u$ and $v$ in the sense that $(u\circ v)_i:=u_i v_i$, $i=1,\ldots,m$. A nice feature of this over-parameterized model is that its objective function is a polynomial (and so ${\cal C}^2$-smooth and subanalytic), and it satisfies the {\em strong strict saddle point property}; see \cite[Appendix C]{poon} for the detailed derivation, and \cite{tk24,Poon} for more information.

To proceed, we first use Theorem~\ref{P3.3} to see that the generalized metric subregularity holds for $\nabla f_{OP}$ with respect to $(\psi,\Theta)$, where $\psi(t)=t^\gamma$ for some $\gamma>0$. In fact, it is possible to identify the exponent $\gamma$ explicitly. As shown in \cite{tk24}, the fulfillment of the {\em strict complementarity condition} 
\begin{equation}\label{eq:strict_CP}
0 \in 2 A^T(A \overline{x}-b) + {\rm ri} \,\big(  \nu \, \partial \|\cdot\|_1 (\overline{x})\big),
\end{equation} 
where `${\rm ri}$' denotes the relative interior, ensures that $f_{OP}$ satisfies the KL property at any $\overline{x} \in \Gamma$ with the KL exponent $1/2$. In the absence of the strict complementarity condition, it is shown in \cite{tk24} that $f_{OP}$ satisfies the KL property at any $\overline{x} \in \Gamma$ with the KL exponent $3/4$.

Therefore, Corollary~\ref{coro3.1}, tells us that for any $\overline{x} \in \Theta$ satisfying the strict complementarity condition \eqref{eq:strict_CP}, $\nabla f_{OP}$ enjoys the standard metric subregularity property (as $\psi(t)=t$) at $\overline{x}$ with respect to $\Theta$. In the absence of strict complementarity condition, $\nabla f_{OP}$ satisfies the generalized metric subregularity property at $\overline{x}$ with respect to $(\psi,\Theta)$, where $\psi(t)=t^3$.}
\end{example}

\begin{example}[\bf best rank-one matrix approximation]\label{ex-rank} {\rm Consider the following function $f:\mathbb{R}^m \rightarrow \mathbb{R}$ that appears in the {\em rank-one matrix approximation}
\begin{equation*}
f(x):=\frac{1}{2}\|xx^T-M\|^2_{F}=\frac{1}{2}\sum_{i,j=1}^m (x_ix_j-M_{ij})^2,
\end{equation*}
where we suppose for simplicity that $M$ is a symmetric $(m\times m)$ matrix with its maximum eigenvalue $\lambda_{\max}(M)>0$.
Direct verification shows that
\begin{equation*}
\nabla f(x)=\|x\|^2x-Mx \mbox{ and } \nabla^2 f(x)=2xx^T-M+\|x\|^2 I_m,
\end{equation*}
and hence we have $\Gamma=\{x\in\mathbb{R}^m\;|\;Mx=\|x\|^2 x\}$. Let us check next that $\Theta = \Omega$, where 
\begin{eqnarray*}
\Omega:=\big\{x\in\mathbb{R}^m\;\big|\; x \mbox{ is an eigenvector of } M \mbox{ associated with the largest eigenvalue} \mbox{ and } \|x\|=\sqrt{\lambda_{\max}(M)}\big\}.
\end{eqnarray*} 
Indeed, take $x \in \Theta$ and observe that if $x=0$, then $\nabla^2 f(x)=-M$ having a negative eigenvalue. This is impossible by $x \in \Theta$. For $x \neq 0$, $x$ is an eigenvector of $M$ with the associated eigenvalue $\lambda=\|x\|^2$. Suppose that $\lambda<\lambda_{\max}(M)$ and get, for any unit eigenvector $d$ associated with $\lambda_{\max}(M)$, that
\begin{equation*}
0\le d^T\nabla^2 f(x)d = 2(x^Td)^2 -d^TMd +\|x\|^2 \|d\|^2=- \lambda_{\max}(M) + \lambda.
\end{equation*}
This is also impossible, and so $\lambda=\lambda_{\max}(M)$ implying that $\|x\|=\sqrt{\lambda_{\max}(M)}.$ Thus $\Theta \subset \Omega$.

Now take $x \in \Omega$ and see that for any $d \in \mathbb{R}^m$ we have
\begin{eqnarray*}
d^T(\nabla^2 f(x))d &=& 2 (x^Td)^2-d^TMd+\|x\|^2 \|d\|^2 \\
& \ge & 2(x^Td)^2 - \lambda_{\max}(M) \|d\|^2 +\lambda_{\max}(M) \|d\|^2 \ge 0,
\end{eqnarray*}
which shows that the reverse inclusion holds, and the claim follows.

Suppose next that the largest eigenvalue $\lambda_{\max}(M)$ is simple, i.e., its geometric multiplicity is one; note that this condition holds generically. Then the eigenspace of $M$ associated with $\lambda_{\max}(M)$ is one-dimensional, and so $\Theta$ consists only of two elements $\{\pm v\}$, where $v$ is the eigenvector associated with the largest eigenvalue, and of length $\sqrt{\lambda_{\max}(M)}$. Both of these two elements are global minimizers of $f$. Hence in this case, $f$ satisfies the strong strict saddle point property. Furthermore, it follows from Theorem~\ref{P3.3} that the generalized metric subregularity holds for $\nabla f$ with respect to $(\psi,\Theta)$, where $\psi(t)=t^\gamma$ for some $\gamma>0$. Indeed, direct computations show that at $\bar x=\pm v$, where $v$ is the eigenvector associated with the largest eigenvalue, we have $\nabla^2 f(\bar x)=2vv^T-M+\|v\|^2 I_m=2vv^T-M+ \lambda_{\max}(M) I_m$. This tells us that $d^T\nabla^2 f(\bar x)d \ge 0$ for all $d \in \mathbb{R}^m$, and that $d^T\nabla^2 f(\bar x)d = 0$ if and only if $v^Td=0$ and $Md=\lambda_{\max}(M)d$, which happens only when $d=0$. Therefore, $\nabla^2 f(\bar x) \succ 0$, and so the metric subregularity holds for $\nabla f$ with respect to $\Theta$ at any $\overline{x} \in \Theta$.}
\end{example}

\begin{example}[\bf generalized phase retrieval problem]\label{ex:retr} {\rm Consider the optimization problem
$$
\min_{x \in \mathbb{R}^m} f(x)= \sum_{i=1}^n \varphi\big(g_i(x)\big),
$$
where $\varphi: \mathbb{R} \rightarrow \mathbb{R}$ is a ${\cal C}^2$ strictly convex subanalytic function such that $\varphi'(0)=0$, and where $g_i:\mathbb{R}^m \rightarrow \mathbb{R}$ are defined by
$g_i(x):=\frac{1}{\gamma_i}[(a_i^Tx)^2-b_i^2]$
with $\gamma_i >0$, $a_i \in \mathbb{R}^m \backslash \{0\}$, and $b_i \in \mathbb{R} \backslash \{0\}$, $i=1,\ldots,n$. We consider the two {\em typical choices} for $\gamma_i$ as $\gamma_i=1$ and $\gamma_i=|b_i|$, $i=1,\ldots,n$, and the two typical choices for $\varphi$ given by:
\begin{itemize}
\item the {\em polynomial loss function} $\varphi_{PL}(t)=t^{2p}$ with $p \ge 1$;
\item the {\em pseudo-Huber loss function} in the form
$\varphi_{PHL}(t)= \delta^2 \bigg(\sqrt{1+\frac{t^2}{\delta^2}}-1 \bigg)$ with $\delta>0$.
\end{itemize}
This class of problems arises in {\em phase retrieval models} (see \cite{Cai2022,Cai2023,yzs}) when the goal is to recover an unknown signal from the observed magnitude of some linear transformation such as, e.g., the Fourier transform.

Below we consider the {\em two settings}: {\bf(I)} the vectors $\{a_1,\ldots,a_n\}$ are linearly independent; {\bf(II)} the vectors
$\{a_1,\ldots,a_n\}$ are independent and identically distributed (i.i.d.) standard Gaussian random vectors, $b_i=a_i^Tu$ for some {$u \in \mathbb{R}^m \setminus\{0\}$} and $\varphi(t)=t^2$. Note that in the first setting, the linear independence condition can be easily verified, and it is usually satisfied  when $m$ is much larger than $n$. The second setting is also frequently considered in the literature; see, 
e.g., \cite{Cai2022,Cai2023,yzs}.  

First we consider setting (I), where $\{a_1,\ldots,a_n\}$ is linearly independent, to verify that $\nabla f$ satisfies the generalized metric subregularity property with respect to the second-order stationary set $\Theta$. Observe that
\begin{eqnarray*}
\nabla f(x) &=& \sum_{i=1}^n \varphi'(g_i(x)) \nabla g_i(x)  =  \sum_{i=1}^n \frac{2}{\gamma_i} \, \big[ (a_i^Tx) \, \varphi'\big(g_i(x)\big) \big] \, a_i,
\end{eqnarray*}
and thus it follows from the linear independence condition that
\begin{equation*}
\Gamma=\big\{x\;\big|\;\nabla f(x)=0\big\}=\bigcap_{i=1}^n\Big\{x\;\Big|\;a_i^Tx=0\;\mbox{or}\;\varphi' \big(g_i(x)\big)=0\Big\},
\end{equation*}
and for all $x \in \mathbb{R}^m$ we have
$\nabla^2 f(x)=  \sum_{i=1}^n \frac{2}{\gamma_i} \, \big\{ \varphi' \big(g_i(x)\big)  + {\frac{2}{\gamma_i}}(a_i^Tx)^2\,  \varphi''(g_i(x)) \big\} \, a_ia_i^T.$ Let us now show that 
\begin{equation}\label{theta3}
\Theta=\big\{x\;\big|\;\varphi' \big(g_i(x)\big)=0,\;i=1,\ldots,n\big\}.
\end{equation}
The inclusion ``$\supset$" in \eqref{theta3} is obvious. To check the reverse inclusion, pick $x \in \Theta$ and suppose that there exists $i_0$ such that $\varphi' \big(g_{i_0}(x)\big) \neq 0$. By $\Theta \subset \Gamma$,  we have $a_{i_0}^T x=0$, which tells us that $g_{i_0}(x)=(a_{i_0}^T x)^2-b_i^2 =-b_i^2<0$. Since $\varphi'(0)=0$ and $\varphi$ is strictly convex (and hence $\varphi'$ is strictly increasing), it follows that $\varphi'(g_{i_0}(x))< \varphi'(0)=0$. For $d=a_{i_0}$, we then have
$d^T\nabla^2 f(x)d=  {\frac{2}{\gamma_i} }\varphi'(g_{i_0}(x)) \|a_{i_0}\|^{4}<0$, which contradicts the assumption that $x \in \Theta$ and thus justifies \eqref{theta3}. Observe further that
\begin{eqnarray*}
\Theta =\big\{x\;\big|\;\varphi' \big(g_i(x)\big)=0,\;i=1,\ldots,n\big\}
=\big\{x\;\big|\;g_i(x)=0,\;i=1,\ldots,n\big\} 
=\big\{x\;\big|\; |a_i^Tx|= |b_i|,\;i=1,\ldots,n\big\}
\end{eqnarray*}
telling us that $\Theta$ reduces to the set of global minimizers. Therefore, the strong strict saddle point property holds for $f$ if the linear independence condition is satisfied. Then it follows from Theorem~\ref{P3.3} that generalized metric subregularity is fulfilled for $\nabla f$ with respect to $(\psi,\Theta)$, where $\psi(t)=t^\gamma$ for some $\gamma>0$.

Let us further specify the exponent $\gamma$ in both bullet cases above by using Corollary~\ref{coro3.1}. In the {\em polynomial case} $\varphi(t)=t^{2p}$ with $p \ge 1$, the direct verification shows that $|\varphi'(t)| \ge c \, \varphi(t)^{\frac{2p-1}{2p}}$ for some $c>0$. Thus $\varphi$ satisfies the KL property with the KL exponent $1-\frac{1}{2p}$. Let $g(x)=(g_1(x),\ldots,g_n(x))$ for all $x \in \mathbb{R}^m$, where the components $g_i$ are defined above. We deduce from the previous considerations that the inclusion $x \in \Theta$ implies that $a_i^Tx \neq 0$ whenever $i=1,\ldots,n$. Hence for any $\overline{x}\in \Theta$, the derivative operator $\nabla g(\overline{x}): \mathbb{R}^m \rightarrow \mathbb{R}^n$ is surjective by the linear independence assumption. Corollary~\ref{coro3.1} tells us in this case that the
generalized metric subregularity holds for $\nabla f$ with respect to $(\psi,\Theta)$, where $\psi(t)=t^{2p-1}$.

In the case where $\varphi$ is the {\em pseudo-Huber loss}, we have $\varphi''(t)>0$ for all $t \in \mathbb{R}$, and it is not hard to verify that the function $\overline{\varphi}(y):=\sum_{i=1}^n \varphi (y_i)$ satisfies the KL property at $\overline y=(g_1(\overline{x}),\ldots,g_n(\overline{x}))$ with the KL exponent $1/2$. Note also that $f(x):=(\overline \varphi \circ g)(x)$ with $g(x)=(g_1(x),\ldots,g_n(x))$ for all $x \in \mathbb{R}^m$, and so Corollary~\ref{coro3.1} implies that the standard metric subregularity
holds for $\nabla f$ at any $\overline{x} \in \Theta$ with respect to $\Theta$.

Next we consider the second setting (II).  As shown in \cite{Cai2022} in the typical case of $\gamma_i=|b_i|=|a_i^T u|$, there exist positive constants $c,C$ such that if $n \ge  C m$, where $n$ is the number of samples and $m$ is the dimension of the underlying space, then with probability at least $1-e^{-c \,  n}$ the only local minimizers of $f$ are $\pm u$ (which are also global minimizers),  $f$ is strongly convex in a neighborhood of $\pm u$, and all the other critical points are either strict saddle points or local maximizers with strictly negative definite Hessians. Thus in this setting, when the number of samples $n$ is sufficiently large, we get with high probability that $\Theta =\{\pm u\}$, and that for any $\overline{x} \in \Theta$ the strong saddle point property holds at $\overline{x}$. It follows therefore from Corollary~\ref{coro3.1} that the standard metric subregularity is satisfied for $\nabla f$ at $\overline{x}$  with respect to $\Theta$. Note finally that in the case where $\gamma_i=1$, a similar result on the  saddle point property is obtained  in \cite{Cai2023} under the stronger assumption that $n \ge C m \log(m)$ for some positive number $C$.}
\end{example}\vspace*{-0.2in}

\section{High-order Regularized Newton Method with Momentum}\label{sec:alg}
\setcounter{equation}{0}\vspace*{-0.05in}

In this section, we design a new high-order Newtonian method with momentum and justify its fast convergence by using the above results on generalized metric regularity and related investigations.\vspace*{0.05in}

For the optimization problem $\min_{x \in \mathbb{R}^m} f(x)$, the following assumptions are imposed:
\begin{assumption}[\bf major assumptions]\label{A3.1}
The objective function $f$ satisfies the conditions:

\item{\bf(1)} $f$ is ${\cal C}^2$-smooth and bounded below, i.e., $f^*=\inf_{x\in \mathbb{R}^m}f(x)>-\infty$.

\item{\bf(2)} There exists a closed convex set $\mathcal{F} \subset \mathbb{R}^m$ with nonempty interior such that $\mathcal{L}(f(x_0))\subset {\rm int}(\mathcal{F})$, where $x_0\in {\rm int}(\mathcal{F})$ is a starting point of the iterative process. 

\item{\bf(3)} The gradient $\nabla f$ is Lipschitz continuous with modulus $L_1>0$ on $\mathcal{F}$, and
the Hessian of $f$ is H\"{o}lder-continuous on $\mathcal{F}$ with exponent $q$, i.e.,  there exist numbers $L_2>0$ and $q\in (0,1]$ such that
 $\|\nabla^2f(x)-\nabla^2f(y)\|\leq L_2 \|x-y\|^q$ for all $x, y\in \mathcal{F}$.
\end{assumption}

Given $x\in \mathbb{R}^m$ and $q\in (0, 1]$, define the {\em $(q+2)$th-order regularized quadratic approximation} for $f$ at $x$ by
\begin{equation}\label{A88}
f_{\sigma}(y;x):=f(x)+\left\langle\nabla f(x), y-x\right\rangle+\frac{1}{2}\langle \nabla^2 f(x)(y-x), y-x\rangle+\frac{\sigma}{(q+1)(q+2)}\|y-x\|^{q+2},
\end{equation}
where $\sigma>0$ serves as a regularization parameter. Consider also
\begin{equation}\label{A89}
\overline{f}_{\sigma}(x):= \min_{y \in \mathbb{R}^m} f_{\sigma}(y; x) \quad \text{and} \quad y_{\sigma}(x) \in \mathop{\arg\min}_{y \in \mathbb{R}^m} f_{\sigma}(y; x).
\end{equation}

Here is our high-order Newtonian method with momentum steps. 

\begin{algorithm}[H]\label{ALg}
\caption{High-order regularization method with momentum} 
\hspace*{0.02in}
\begin{algorithmic}[htpb]
\State~1: {\bf Input:} initial point\; $x_0=\widehat x_0\in \mathbb{R}^m$,\;scalars\; $\overline{\sigma}\in \left(\frac{2L_2}{q+2}, L_2\right]$\; and\; $\zeta\in [0,1)$.
\State~2: {\bf for} {$k=0, 1, \dots$}\;{\bf do}
\State~3: {\bf Regularization step:} Choose $\sigma_k\in [\overline{\sigma}, 2L_2]$ 
and find
\begin{equation}\label{116}
\widehat x_{k+1}:=y_{\sigma_k}(x_k) \in \mathop{\arg\min}_{y \in \mathbb{R}^m} f_{\sigma_k}(y; x_k).
\end{equation}
\State~4: {\bf Momentum step:}
 \begin{equation}\label{5.025}
\beta_{k+1}=\min\big\{\zeta, \|\nabla f(\widehat x_{k+1})\|, \|\widehat x_{k+1}-x_k\|\big\},
\end{equation}
\begin{equation}\label{5.026}
\widetilde x_{k+1}=\widehat x_{k+1}+\beta_{k+1}(\widehat x_{k+1}-\widehat x_{k}).
\end{equation}
\State~5: {\bf Monotone step:}
\begin{equation}\label{5.27}
x_{k+1}=\mathop{\arg\min}_{x\in \{\widehat x_{k+1}, \widetilde x_{k+1}\}} f(x).
\end{equation}
\State~6: {\bf end for}
\end{algorithmic}
\end{algorithm}

Note that when $q=1$ and $\zeta=0$ (i.e., the Hessian is locally Lipschitzian and there is no momentum step), this algorithm reduces to the classical cubic regularized Newton method proposed in \cite{np06}. In the case where $\zeta=0$ (i.e., no momentum steps are considered), this method has also been considered in \cite{gne} with a comprehensive analysis of its complexity. The superlinear/quadratic convergence for the cubic Newton method was first established in \cite{np06} under the {\em nondegeneracy} condition, and then was relaxed to the {\em uniform metric subregularity/error bound} condition on $\nabla f$ in \cite{yzs}. We also note that, if $L_2$ can be obtained easily, then the parameter $\sigma_k$ in Algorithm 1 can be chosen as $L_2$. Otherwise, the parameter $\sigma_k$ therein can be obtained by using line search strategies outlined in \cite{np06}.  Incorporating momentum steps is a well-adopted technique in first-order optimization to accelerate numerical performances. Therefore, it is conceivable to also consider incorporating the {\em momentum technique} into {\em high-order regularization methods}.

To illustrate this, we see from the following plots (Figures~1 and 2) that incorporating the momentum steps into the {\em cubic Newton methods} in solving phase retrieval and matrix completion problems can lead to faster convergence with comparable quality of numerical solutions. Here we follow the setup as in \cite{yzs} for both problems with dimension $128$, and use their code (available at https://github.com/ZiruiZhou/cubicreg app.git) for implementing the cubic regularized Newton method. For the cubic regularized Newton method with momentum steps, we choose $\xi=0.1$ and set $q=1$ in Algorithm~1 and follow the same approach as in \cite{yzs} to solve the subproblems \eqref{116}. Then we use the same notion of relative error as in \cite{yzs} and plot the relative error (in the log scale) in terms of the number of iterations. \footnote{While it is clear from the presented numerical illustrations that incorporating momentum steps can have numerical advantages, further extensive numerical experiments are definitely needed. This goes beyond the scope of this paper and will be carried out in a separate study.}\vspace*{0.05in}
\begin{figure}[htpb!]
\centering
\begin{minipage}{.53\textwidth}
\centering
\includegraphics[scale=0.17]{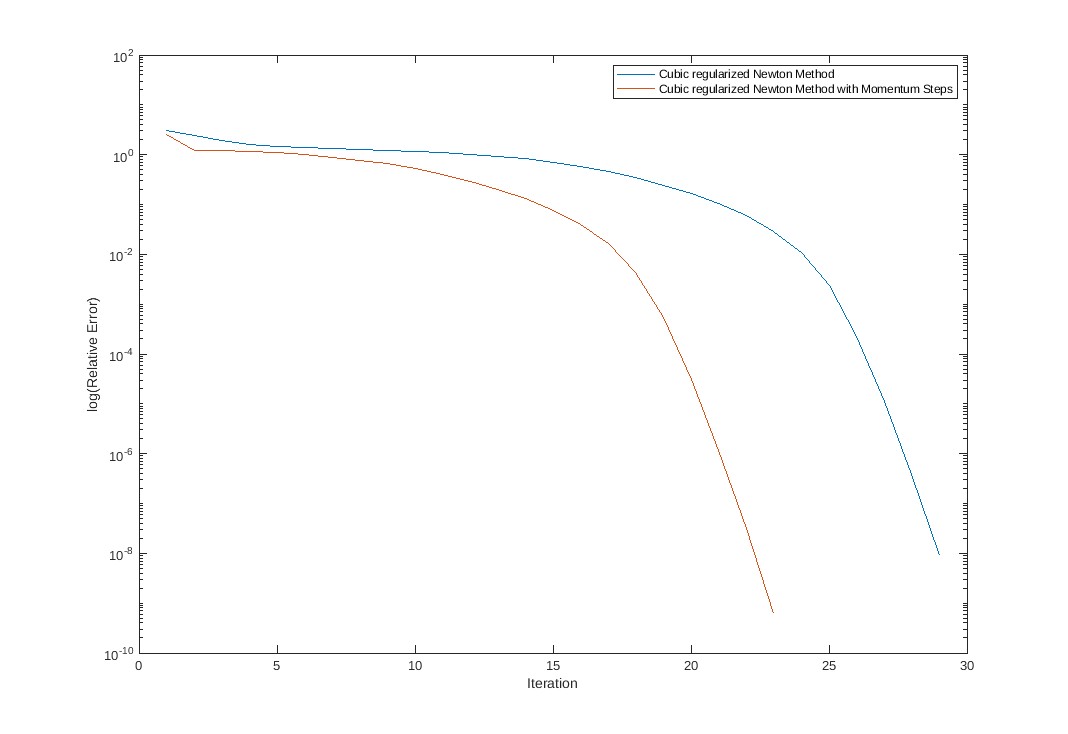}
\captionof{figure}{Phase retrieval problem}
\label{fig:test1}
\end{minipage}%
\begin{minipage}{.53\textwidth}
\centering
\includegraphics[scale=0.3]{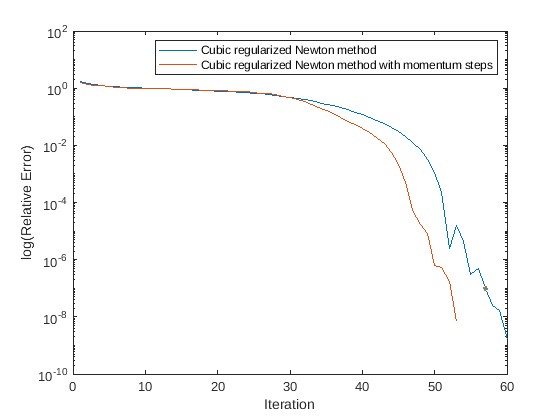}
\captionof{figure}{Matrix completion problem}
\label{fig:test2}
\end{minipage}
\end{figure}

Let us emphasize that, to the best of our knowledge, {\em superlinear/quadratic convergence} for high-order regularized methods with momentum steps {\em has not been considered} in the literature. We also mention that, even when it is specialized to the celebrated cubic regularized Newton method without momentum steps, our assumptions for quadratic convergence are strictly weaker than the ones imposed in \cite{yzs}.

Note that in the regularization step, Algorithm~1 requires computing $y_{\sigma_k}(x_k)$, which is a global minimizer of $f_{\sigma_k}(\cdot;x_k)$. This can be done by solving either an equivalent convex optimization problem, or an eigenvalue problem; see, e.g, \cite{gne,np06,yzs}. The necessary and sufficient optimality conditions for global minimizers of the regularized quadratic approximation \eqref{A88} presented in the next lemma are taken from \cite{hsy}.

\begin{lemma}[\bf optimality conditions for regularized quadratic approximations]\label{L5.2} Given $x\in \mathcal{F}$, we have that $y_{\sigma}(x) $ is a global minimizer of \eqref{A88} if and only if the following conditions hold:
\begin{equation}\label{4.5}
\nabla f(x)+ \nabla^2 f(x)(y_{\sigma}(x)-x)+\frac{\sigma}{q+1}\|y_{\sigma}(x)-x\|^q(y_{\sigma}(x)-x)=0,
\end{equation}
\begin{equation}\label{4.6}
\nabla^2 f(x)+\frac{\sigma}{q+1}\|y_{\sigma}(x)-x\|^qI_m\succeq 0.
\end{equation}
\end{lemma}

For the reader's convenience, we also formulate straightforward consequences of our major requirements in Assumption~\ref{A3.1} that can be found in \cite[(2.7) and (2.8)]{gne}.

\begin{lemma}[\bf consequences of major assumptions]\label{L3.2}
Suppose that all the requirements in Assumption~{\rm\ref{A3.1}} are satisfied. Then for any $x, y\in \mathcal{F}$ we get the estimates
\begin{equation}\label{4.27}
\|\nabla f(y)-\nabla f(x)-\nabla^2f(x)(y-x)\|\leq \frac{L_2}{q+1}\|y-x\|^{q+1},
\end{equation}
\begin{equation}\label{4.33}
\left|f(y)-f(x)-\langle \nabla f(x), y-x\rangle-\frac{1}{2}\langle \nabla^2 f(x)(y-x), y-x\rangle\right|\leq \frac{L_2}{(q+1)(q+2)}\|y-x\|^{q+2}.
\end{equation}
\end{lemma}

The next technical lemma is needed for the convergence justification below. 

\begin{lemma}[\bf relationships for H\"older exponents]\label{L4.4}
Given $ b\in (0, \infty)$ and  $q\in (0,1]$, the following hold. 

{\bf(i)} For each $\alpha>0$, there is a unique solution $t(\alpha) > 2$ to the equation
\begin{equation*}
\alpha t^{q+1}-2\alpha t^q-b(q+1)t-b=0.
\end{equation*}
Moreover, the function $\alpha \mapsto t(\alpha)$  is decreasing on $(0,\infty)$. 

{\bf(ii)} Letting $a\in (0,\infty)$ and $\ell:=t(a) >0$, suppose that
\begin{equation}\label{L98}
a x^{q+1}\leq by^{q+1}+2a x^qy+ b(q+1)x y^q\;\mbox{ for }\;x, y\in \mathbb{R}_+.
\end{equation}
Then for such numbers $x$ and $y$, we have the inequality $x\leq\ell y$.
\end{lemma}

\begin{proof} To verify (i), we define the function
$f(\alpha, t):=\alpha t^{q+1}-2\alpha t^q-b(q+1)t-b$ for all $(\alpha, t)\in (0,\infty)\times \mathbb{R}_+$ and denote $t(\alpha)=\{t\;|\;f(\alpha, t)=0\}$ with $\alpha\in (0,\infty)$, which occurs to be single-valued as shown below. Direct calculations give us the expressions
\begin{equation*}
\frac{\partial f}{\partial t} (\alpha, t)=\alpha(q+1)t^q-2\alpha qt^{q-1}-b(q+1),
\end{equation*}
\begin{equation*}
\frac{\partial^2 f}{\partial t^2}(\alpha, t)=\alpha q(q+1)t^{q-1}-2\alpha q(q-1)t^{q-2}=\alpha qt^{q-2}\big((q+1)t-2(q-1)\big)> 0,\quad t\in (0,\infty),
\end{equation*}
and hence $ \frac{\partial f}{\partial t}(\alpha, \cdot)$ is increasing on $\mathbb{R}_+$. Combining this with the facts that $\lim_{t \to 0^+} \frac{\partial f}{\partial t} (\alpha, t)<0$ and $\lim_{t \to \infty}  \frac{\partial f}{\partial t} (\alpha, t)=\infty$  for any $\alpha\in (0,\infty)$ guarantees the existence of $\tilde t(\alpha)>0$ such that $f(\alpha, \cdot)$ is decreasing on $(0, \tilde t(\alpha)]$ and increasing on $(\tilde t(\alpha), \infty)$. Observing that $f(\alpha, 0)=-b<0$ and $\lim_{t\to \infty} f(\alpha, t)=\infty$  for any $\alpha\in (0,\infty)$, we see that the mapping $\alpha \mapsto t(\alpha)$ with $\alpha\in (0,\infty)$  is single-valued and satisfies the implication
\begin{equation}\label{L02}
f(\alpha, t)\leq 0\Longrightarrow t\leq t(\alpha).
\end{equation}
 Defining further the univariate function
\begin{equation*}
\alpha(t):=\frac{b+b(q+1)t}{t^{q+1}-2t^q}\quad\mbox{for all }\;t\in (2, \infty),
\end{equation*}
we claim that $t\mapsto\alpha(t)$ is decreasing on $(2,\infty)$ with the range $(0,\infty)$. This clearly ensures that its inverse function $\alpha \mapsto t(\alpha)$ is decreasing on $(0,\infty)$ with a range of $(2,\infty)$, and so (i) follows.

Thus to prove (i), it remains to verify the formulated claim. Direct calculations show that 
\begin{equation}\label{L201}
\begin{aligned}
\alpha'(t)=&\frac{b(q+1)(t^{q+1}-2t^q)-(b+b(q+1)t)\Big((q+1)t^q-2qt^{q-1}\Big)}{(t^{q+1}-2t^q)^2}\\
=&\frac{bt^{q-1}}{(t^{q+1}-2t^q)^2}\Big[(q+1)(t^{2}-2t)-(1+(q+1)t)\Big((q+1)t-2q\Big)\Big]\\
=&\frac{bt^{q-1}}{(t^{q+1}-2t^q)^2}\Big[-q(q+1)t^2+(2q^2-q-3)t+2q\Big].
\end{aligned}
\end{equation}
For all $t\in\mathbb R$, we consider the function
$h(t):=-q(q+1)t^2+(2q^2-q-3)t+2q$ and observe that
$\lim_{t\to -\infty}h(t)=-\infty$ and $h(0)=2q>0$ and $h(2)=-4q-6<0$, which allows us to deduce from \eqref{L201} that $\alpha(t)$ is decreasing on $(2,\infty)$. Combining this with the facts that $\lim_{t\to 2}\alpha(t)=\infty$ and $\lim_{t\to \infty}\alpha(t)=0$ justifies the above claim and hence the entire assertion (i).

To verify now assertion (ii), we deduce from \eqref{L98} that $a x^{q+1}- by^{q+1}-2a x^q y- b(q+1)x y^q\leq 0$.
Letting $x:=ty$ for some $t\in \mathbb{R}_+$ and substituting it into the above inequality tells us that
\begin{equation}\label{L26}
\left(a t^{q+1}-2a t^q- b(q+1)t- b \right)y^{q+1}\leq 0.
\end{equation}
If $y=0$, then \eqref{L98} implies that $x=0$, and so (ii) trivially holds. To justify the inequality $x\le\ell y$, it suffices to consider the case where $y>0$.  Dividing both sides of \eqref{L26} by $y^{q+1}$, we get
\begin{equation*}
f(a, t)=a t^{q+1}-2a t^q- b(q+1)t- b\leq 0,
\end{equation*}
and therefore the claimed inequality follows from \eqref{L02} and the fact that $t(\alpha) > 2$. The proof is complete.
\end{proof}

The following lemma is an important step to justify the desired performance of Algorithm~1. 

\begin{lemma}[\bf estimates for regularized quadratic approximations]\label{L3.3} Under Assumption {\rm\ref{A3.1}}, let $x\in \mathcal{F}$, and let $\widecheck x$ be a projection point of $x$ to $\Theta$. If $\widecheck x\in \mathcal{F}$, then for any $\sigma>0$ we have the estimate
\begin{equation}\label{A98}
 \|y_{\sigma}(x)-x\|\leq \ell_\sigma \,  d(x, \Theta),
\end{equation}
where $y_{\sigma}(x)$ is a minimizer of the regularized quadratic approximation from \eqref{A88} and \eqref{A89} with $\ell_\sigma:=t^*$ defined as the unique positive solution $t^*$ of the equation $H_{\sigma}(t)=\sigma t^{q+1}-2\sigma t^q-L_2(q+1)t-L_2=0$ with $L_2$ taken from part \rm(3) of Assumption~{\rm\ref{A3.1}}.
\end{lemma}
\begin{proof}
Denote $z:=y_{\sigma}(x)$ and deduce from condition \eqref{A89} and 
Lemma~\ref{L5.2} that
\begin{equation}\label{A92}
0=\nabla f(x)+\nabla^2f(x)(z-x)+\frac{\sigma}{q+1} \|z-x\|^{q}(z-x).
\end{equation}
By $\widecheck x\in \Theta$, we have $\nabla f(\widecheck x)=0$ and $\nabla^2f(\widecheck x)\succeq 0$. Then it follows from \eqref{A92} that
\begin{equation*}
\begin{aligned}
&\left(\nabla^2 f(\widecheck x)+\frac{\sigma}{q+1} \|z-x\|^{q}I_m\right)(z-\widecheck x)
=\nabla^2f(\widecheck x)\Big((z-x)-(\widecheck x-x)\Big)+\frac{\sigma}{q+1} \|z-x\|^{q}\Big((z-x)-(\widecheck x-x)\Big)\\
&=\nabla f(\widecheck x)-\nabla f(x)-\nabla^2 f(\widecheck x)(\widecheck x-x)-\frac{\sigma}{q+1} \|z-x\|^{q}(\widecheck x-x)
-(\nabla^2f(x)-\nabla^2f(\widecheck x))(z-x).
\end{aligned}
\end{equation*}
Combining the latter with $\nabla^2f(\widecheck x)\succeq 0$ gives us
\begin{equation*}
\left\|\left(\nabla^2 f(\widecheck x)+\frac{\sigma}{q+1} \|z-x\|^{q}I_m\right)(z-\widecheck x)\right\|\geq \frac{\sigma}{q+1} \|z-x\|^{q}\|z-\widecheck x\|,
\end{equation*}
and then employing $\widecheck x\in \mathcal{F}$, Lemma~\ref{L3.2}, and Assumption~\ref{A3.1}(3), we arrive at
\begin{equation}\label{A96}
\begin{aligned}
\frac{\sigma}{q+1} \|z-x\|^{q}\|z-\widecheck x\|\leq & \|\nabla f(x)-\nabla f(\widecheck x)-\nabla^2 f(\widecheck x)(x-\widecheck x)\|+\frac{\sigma}{q+1} \|z-x\|^{q}\|x-\widecheck x\|\\
&+\|\nabla^2f(x)-\nabla^2f(\widecheck x)\|\cdot \|z-x\|\\
\leq& \frac{L_2}{q+1} \|x-\widecheck x\|^{q+1}+\frac{\sigma}{q+1} \|z-x\|^{q}\|x-\widecheck x\|+L_2\|z-x\| \cdot \|x-\widecheck x\|^q.
\end{aligned}
\end{equation}
Applying further the triangle inequality $\|z-\widecheck x\|\geq \|z-x\|-\|x-\widecheck x\|$ ensures that
\begin{equation*}
\frac{\sigma}{q+1} \|z- x\|^{q+1}-\frac{\sigma}{q+1} \|z-x\|^{q}\|x-\widecheck x\|\leq \frac{\sigma}{q+1} \|z-x\|^{q}\|z-\widecheck x\|,
\end{equation*}
which being substituted into \eqref{A96} yields
\begin{equation}\nonumber
\frac{\sigma}{q+1} \|z- x\|^{q+1}\leq \frac{L_2}{q+1} \|x-\widecheck x\|^{q+1}+\frac{2\sigma}{q+1} \|z-x\|^{q}\|x-\widecheck x\|+L_2\|z-x\| \cdot \|x-\widecheck x\|^q.
\end{equation}
By Lemma~\ref{L4.4}(ii), this allows us to find a positive scalar  $\ell_\sigma$ (depending only on $\sigma$) such that
$\|z-x\|\leq \ell_\sigma \|x-\widecheck x\|$. Recalling that $z=y_{\sigma}(x)$ and $d(x, \Theta)=\|x-\widecheck x\|$ verifies the desired inequality \eqref{A98}.
\end{proof}

The last lemma here plays a significant role in the convergent analysis of our algorithm. Its proof is rather technical and mainly follows the lines of justifying fast convergence of the cubic regularization method in \cite{yzs}  with  additional adaptations to accommodate the momentum steps and the H\"{o}lderian continuity of Hessians. For completeness and the  reader's convenience, we present details in the Appendix (Section~\ref{appe}).

\begin{lemma}[\bf estimates for iterates of the high-order regularized method with momentum]\label{L3.1} Let $\{x_k\}$ be a sequence of iterates in Algorithm~{\rm 1} under Assumption~{\rm\ref{A3.1}}. If the set $\mathcal{L}(f(x_{k_0}))$ is  bounded for some $k_0\in\mathbb N$, then the following assertions hold. 

{\bf(i)} For all $k\in\mathbb N$ with $k\ge k_0$, we have the inequalities
\begin{equation}\label{5.38}
f(x_{k+1})+\frac{(q+2)\overline{\sigma}-2L_2}{2(q+1)(q+2)}\|\widehat x_{k+1}-x_k\|^{q+2}\leq  f(x_k),
\end{equation}
and the limit $v:=\lim_{k\to \infty} f(x_k)$ exists.

{\bf(ii)} For all $k\in\mathbb N$ with $k\ge k_0$, the relationships 
\begin{equation}\label{632}
\begin{aligned}
\|x_{k+1}-x_k\|\leq (1+\gamma)\|\widehat x_{k+1}-x_k\|, \quad \lim_{k\to \infty}\|\widehat x_{k+1}-x_k\|=0,
\end{aligned}
\end{equation}
\begin{equation}\label{0125}
\|\nabla f(x_{k+1})\|\leq \frac{3(1+\gamma L_1)L_2}{q+1}\|\widehat x_{k+1}-x_k\|^{q+1}
\end{equation}
are fulfilled, where the number $\gamma>0$ in \eqref{632} is defined by
\begin{equation}\label{5.54}
\gamma:=\frac{1}{1-\zeta} \left[\frac{2(q+1)(q+2)}{(q+2)\overline{\sigma}-2L_2} (f(x_0)-f^*)\right]^{\frac{1}{q+2}}.
\end{equation}

{\bf(iii)} The set $\Omega$ of accumulation points of $\{x_k\}$ is nonempty. Moreover, for any $\overline{x}\in \Omega$ we have
\begin{equation}\label{5.53}
f(\overline{x})=v, \quad \nabla f(\overline{x})=0,\; \mbox{ and }\; \nabla^2f(\overline{x})\succeq 0.
\end{equation}
\end{lemma}

Now we are in a position to establish the main result on the performance of Algorithm~1.

\begin{theorem}[\bf convergence rate of the high-order algorithm under generalized metric subregularity]\label{T5.1} Let the iterative sequence $\{x_k\}$ be generated by Algorithm~{\rm 1}, and let $\ox$ be its accumulation point provided that the set $\mathcal{L}(f(x_{k_0}))$ is bounded for some $k_0\in\mathbb N$. In addition to Assumption~{\rm\ref{A3.1}}, suppose that there exists $\eta>0$ such that the generalized metric subregularity condition
\begin{equation}\label{3.118}
\psi\big(d(x,\Theta))\le\|\nabla f(x)\|\;\mbox{ for all }\; x\in B_{\mathbb{R}^m}(\overline{x}, \eta)
\end{equation}
holds with respect to $(\psi,\Theta)$, where $\psi\colon\mathbb{R}_+\to\mathbb{R}_+$ is an admissible function, and where $\Theta$ is the set of second-order stationary points \eqref{2stat}. Define the function $\tau:\mathbb{R}_+\to\mathbb{R}_+$ by 
\begin{equation}\label{558}
\tau(t):=\psi^{-1}\Big(\frac{3(1+\gamma L_1)L_2\ell^{q+1}}{q+1} t^{q+1}\Big),\quad t\in \mathbb{R}_+,
\end{equation}
where $\ell:=\ell_{\overline{\sigma}}$ with $\ell_{\sigma}$ taken from Lemma~{\rm\ref{L3.3}}, and where $\gamma$ is given in \eqref{5.54}. If $\limsup_{t\to 0^+}\frac{\tau(t)}{t}<1$, then $\overline{x}\in \Theta$ and the sequence $\{x_k\}$ converges to $\overline{x}$ as $k\to \infty$ with the convergence rate
\begin{equation}\nonumber
\limsup_{k\to \infty}\frac{\|x_{k}-\overline{x}\|}{{\tau}(\|x_{k-1}-\overline{x}\|)}<\infty.
\end{equation}
\end{theorem}

\begin{proof} By Lemma~\ref{L3.1}, we have the inclusion $\Omega\subset \Theta$ together with  all the conditions in \eqref{5.38}--\eqref{0125}.  Let us now verify the estimate
\begin{equation}\label{1900}
\|\widehat x_{k+1}-x_k\|\leq \ell d(x_k, \Theta)
\end{equation}
for any $k\in\mathbb N$ sufficiently large. Having this, we arrive at all of the claimed conclusions of the theorem by applying the abstract convergence results of Theorem~\ref{P5.2} with $\beta(t):=t^{q+1}$ and $e_k:=\|\widehat x_{k+1}-x_k\|$.

To justify \eqref{1900}, let $\widecheck x_k$ be a projection point of $x_k$ to $\Theta$. Thus we deduce from Lemma~\ref{P10}(iii) that
\begin{equation}\label{123}
\lim_{k\to \infty}d(x_k, \Theta)=\lim_{k\to \infty}\|x_k-\widecheck x_k\|=0.
\end{equation}
Moreover, observing that $x_k\in \mathcal{L}(f(x_0))\subset {\rm int}(\mathcal{F})$ for all $k\in\mathbb N$ and  that $\{x_k\}$ remains bounded allows us to find a compact set $\mathcal{C}\subset {\rm int}(\mathcal{F})$ such that $\{x_k\}\subset \mathcal{C}$. Based on $\{x_k\}\subset \mathcal{C}$ and the conditions given by \eqref{123}, we get that $\lim_{k\to \infty}d(\widecheck x_k, \mathcal{C})=0$. Coupling this with the inclusion $\mathcal{C}\subset {\rm int}(\mathcal {F})$ and the compactness of $\mathcal{C}$ indicates that $\widecheck x_k\in {\rm int}(\mathcal{F})$ for all large  $k$. Consequently, there exists $k_1\geq 1$ such that
$x_k, \widecheck x_k\in \mathcal{F}$ whenever $k\geq k_1$. Employing
Lemmas~\ref{L4.4} and \ref{L3.3} together with $\sigma_k\geq \overline{\sigma}>0$ for all $k$ gives us the number $\ell:=\ell_{\overline{\sigma}}$ from Lemma~\ref{L3.3} ensuring the estimate
\begin{equation*}
\|\widehat x_{k+1}-x_k\|\leq \ell d(x_k, \Theta)\;\mbox{ for all large }\;k\in\mathbb N,
\end{equation*}
which justifies \eqref{1900} and thus completes the proof of the theorem.
\end{proof}

As a consequence of Theorem~\ref{T5.1}, we show now that if $\nabla f$ enjoys the (pointwise) exponent metric subregularity property with respect to the second-order stationary set $\Theta$, then the sequence generated by our algorithm converges to a second-order stationary point at least {\em superlinearly} with an explicit convergence order under a suitable condition on the exponent.

\begin{corollary}[\bf superlinear convergence under exponent metric subregularity]\label{corT5.1} Let $f$ satisfy the Assumption~{\rm\ref{A3.1}} with the $q$-th order H\"{o}lder continuous Hessian for some $q \in (0,1]$, let $\{x_k\}$ be generated by 
Algorithm~{\rm 1}, and let the set  $\mathcal{L}(f(x_{k_0}))$  be bounded for some $k_0\in\mathbb N$. Taking $\overline{x}\in \Omega$ and $p>0$, suppose that there exist $c,\eta>0$ such that
\begin{equation*}
d(x, \Theta)^{p}\leq   c\, \|\nabla f(x)\|\;\mbox{ for all }\;x\in B_{\mathbb{R}^m}(\overline{x}, \eta).
\end{equation*}
If $p<q+1$, then $\{x_k\}$ converges at least superlinearly to $\overline{x}\in \Theta$ as $k\to \infty$ with the rate $\frac{q+1}{p}$, i.e.,
\begin{equation*}
\limsup_{k\to \infty}\frac{\|x_{k}-\overline{x}\|}{\|x_{k-1}-\overline{x}\|^{\frac{q+1}{p}}}<\infty.
\end{equation*}
\end{corollary}
\begin{proof}
Choosing $\psi(t):= c^{-1} t^p$, we have $\tau(t)=O(t^{\frac{q+1}{p}})$ and deduce the conclusion from Theorem~\ref{T5.1}.
\end{proof}

Following the discussions in Remark~\ref{remark:4.4}, we can establish global convergence of the sequence $\{x_k\}$ under the KL property and its convergence rate depending on the KL exponent. For brevity, the proof is skipped. 

\begin{proposition}[\bf convergence under the  KL property]\label{corT6.2}
Let $f$ satisfy Assumption~{\rm\ref{A3.1}} with the $q$-th order H\"{o}lder continuous Hessian for some $q \in (0,1]$. Let $\{x_k\}$ be generated by Algorithm~{\rm 1}, and  let the set $\mathcal{L}(f(x_{k_0}))$  be bounded for some $k_0\in\mathbb N$. Take $\overline{x}\in \Omega$ and suppose that $f$ satisfies the KL property at $\overline{x}$. Then the sequence $\{x_k\}$ converges to $\overline{x}$, which is a second-order stationary point of $f$. If furthermore $f$ satisfies the KL property at $\overline{x}$ with the KL exponent $\theta$, then we have the following convergence rate:

$\bullet$ If $\theta=\frac{q+1}{q+2}$, then the  sequence $\{x_k\}$ converges linearly.

$\bullet$ If $\theta>\frac{q+1}{q+2}$, then the  sequence $\{x_k\}$ converges sublinearly with the rate $O(k^{-\frac{(q+1)(1-\theta)}{(q+2)\theta-(q+1)}})$.

$\bullet$ If $\theta<\frac{q+1}{q+2}$, then the  sequence $\{x_k\}$ converges superlinearly with the rate $\frac{q+1}{(q+2)\theta}$. 
\end{proposition}

\begin{remark}[\bf comparison with known results]\label{remark:5.1} {\rm To the best of our knowledge, the derived (superlinear/linear/sublinear) convergence rates are new in the literature for high-order regularization methods with momentum steps {\em without any convexity assumptions}. To compare in more detail, consider the case where $p=1$, i.e., the metric subregularity of $\nabla f$ at $\overline{x}$ holds with respect to the second-order stationary set $\Theta$. In this case, when applying the high-order regularization methods with momentum steps to a ${\cal C}^2$-smooth function $f$ with the $q$th-order H\"{o}lder continuous Hessian, we obtain the {\em superlinear convergence} of $\{x_k\}$ with rate $q+1$ to a second-order stationary point. Suppose  further that $q=1$ and no momentum step is used. In that case, we recover the {\em quadratic convergence} rate to a second-order stationary point for the cubic regularized Newton method, but under the weaker assumption of the {\em pointwise metric subregularity} at $\overline{x}$ instead of the stronger error bound condition used in \cite{yzs}}.
\end{remark}

Next we provide two examples with {\em explicit degeneracy} illustrating that our convergence results go beyond the current literature on regularized Newton-type methods of optimization.

\begin{example}[\bf fast and slow convergence without the error bound condition]\label{exa:noEB}
{\rm Consider the function $f(x):=(w^Tx-a)^2(w^Tx-b)^{2p}$ with $x \in \mathbb{R}^m$, $w \in \mathbb{R}^m \backslash \{0\}$, $0 \le a <b$, and $m,p \ge 2$. Direct calculations give us the exact expressions for the first-order stationary set $\Gamma=\bigcup_{r \in \{a,b,\frac{b+pa}{1+p}\}}\big\{x\;\big|\;w^Tx=r\big\}$ and the second-order stationary set $\Theta=\bigcup_{r \in \{a,b\}}\big\{x\;\big|\;w^Tx=r\}$. For any $\overline{x} \in \Theta$, we see that $\nabla^2 f(\overline{x})$ has a zero eigenvalue. Similarly to Example~\ref{example:3.3}, it is easy to check that the {\em uniform metric subregularity/error bound condition} for $\nabla f$ with respect to $\Theta$ {\em fails} in this case.  

Let $\overline{x} \in \Theta=\Theta_1 \cup \Theta_2$, where $\Theta_1=\{x\;|\;w^Tx=a\}$ and $\Theta_2=\{x\;|\; w^Tx=b\}$. If $\overline{x} \in \Theta_1$, then there exists $C>0$ such that $d(x,\Theta)=d(x,\Theta_1) \le C \, \|\nabla f(x)\|$ for all $x \in \mathbb{B}_{\mathbb{R}^m}(\overline{x},(b-a)/4)$, and so the metric subregularity condition of $\nabla f$ holds at $\overline{x}$ with respect to $\Theta$. On the other hand, direct verification shows that $f$ satisfies the KL property at $\overline{x}$ with the KL exponent $1-\frac{1}{2p}$ for $\overline{x} \in \Theta_2$. 

Let $\mathcal{F}$ be a compact set, and let $x_0 \in \mathcal{F}$ be such that Assumption~\ref{A3.1} is satisfied. Then we see that the Hessian of $f$ is Lipschitz continuous on $\mathcal{F}$, i.e., $q=1$. This yields the convergence $x_k\to\overline{x} \in \Theta$ of the iterative sequence in Algorithm~1. If $\overline{x} \in \Theta_1$, then Corollary~\ref{corT5.1} implies that the convergence rate is {\em quadratic}. If $\overline{x} \in \Theta_2$, then it follows from Proposition~\ref{corT6.2} that $\{x_k\}$ converges to $\overline{x}$ {\em sublinearly} with the rate $O(k^{-\frac{2}{2p-3}})$.}
\end{example} 

\begin{example}[\bf linear and sublinear convergence under the KL property]\label{exa:linKL} {\rm For $\gamma >2$ and $M>0$, define the ${\cal C}^2$-smooth function $f: \mathbb{R} \rightarrow \mathbb{R}$ by 
\begin{equation*}
f(x):=|\min\{\max\{x,0\},x+M\}|^{\gamma}=\left\{\begin{array}{ccl}
x^{\gamma} & \mbox{ if } & x>0, \\
0& \mbox{ if } & x \in [-M,0], \\
|x+M|^{\gamma} & \mbox{ if } & x<-M.
\end{array}
\right.
\end{equation*}
We have that $\Gamma=\Theta=[-M,0]$ and that $f$ satisfies the KL property at any $\overline{x} \in \Theta$ with the KL exponent $\theta=1-\frac{1}{\gamma}$. On the other hand, the error bound property of $\nabla f$ with respect to $\Theta$ {\em fails} here. 

Let $\mathcal{F}$ be a compact set, and let $x_0 \in \mathcal{F}$ be such that Assumption~\ref{A3.1} is satisfied. It can be easily verified that $f''$ is $q$th-order H\"{o}lder continuous on $\mathcal{F}$ with $q=\min\{\gamma-2,1\}$. If $\gamma \in (2,3]$, then $\theta=\frac{q+1}{q+2}$, and the sequence of iterates $\{x_k\}$ of Algorithm~1 converges to $\overline{x} \in \Theta$ {\em linearly} by Proposition~\ref{corT6.2}. If $\gamma>3$, then $\theta>\frac{q+1}{q+2}$, and so $\{x_k\}$ converges to $\overline{x} \in \Theta$ {\em sublinearly} with the rate $O(k^{-\frac{2}{\gamma-3}})$. Observe that for $\gamma=3$, the {\em linear convergence} can be indeed {\em achieved}. In this case, we have $q=1$ and $L_2=6$, and then take $\sigma_k:=\sigma \in\big(\frac{2L_2}{q+2}, 2L_2\big] =(4,12]$ for all $k$ and $x_0>0$. Assuming without loss of generality that $\{x_k\}$ is an infinite sequence, we get by the construction of Algorithm~1 that
\begin{equation*}
x_{k+1} \in {\rm argmin}_{y}\big\{ f'(x_k)(y-x_k)+ \frac{1}{2}f''(x_k)(y-x_k)^2+ \frac{\sigma}{6}|y-x_k|^3\big\}.
\end{equation*}
Letting $t:=x_{k+1}-x_k$ gives us $f'(x_k)  + f''(x_k) t + \frac{\sigma}{2} t |t| =0$. Suppose now that $x_k>0$, and then $f'(x_k)>0$ and $f''(x_k)>0$. This implies that $t<0$ and hence $3x_k^2  + 6x_k t - \frac{\sigma}{2} t^2 =0$. Therefore, $t=(\frac{6 }{\sigma}- \frac{\sqrt{6 \sigma+36}}{\sigma}) x_k$ and so $x_{k+1}= \rho x_k>0$ with {$\rho=1-\frac{\sqrt{6 \sigma+36}-6}{\sigma} \in (0,1)$}. Arguing by induction tells us that 
$x_k>0$ and $x_{k+1}=\rho x_k$ for all $k\in\mathbb N$. This yields $x_k \rightarrow \overline{x}=0 \in \Theta$, which confirms that the convergence rate is linear.} 
\end{example}

Finally in this section, we specify the convergence results for the case where the {\em strict saddle point property} holds, which allows us to establish stronger convergence results. Observe to this end that if the objective function $f$ satisfies the strict saddle point property (resp.\ strong strict saddle point property), then Algorithm~1 indeed converges to a local minimizer (resp.\ global minimizer) of $f$ with the corresponding {\em superlinear rate}. In particular, this covers the situations discussed in the illustrative examples of Section~\ref{sec:2nd}.\vspace*{0.05in}

The next corollary addresses a concrete case of the over-parameterized compressed sensing model. In this case, Algorithm~1 converges to a {\em global minimizer} with a {\em quadratic convergence rate} under the strict complementary condition. 

\begin{corollary}[\bf quadratic convergence for over-parameterized compressed sensing models]\label{quadratic}
Consider the $\ell_1$-regularization model \eqref{eq:L1} and the associated over-parameterized smooth optimization problem $\min_{x \in \mathbb{R}^{2m}} f(x)$, where $f:=f_{OP}$ as in \eqref{eq:OP}. Then the iterative sequence $\{x_k\}$ of Algorithm~{\rm 1} converges to a global minimizer $\ox$ of \eqref{eq:OP} as $k\to \infty$, and the following assertions hold: 
\begin{itemize}
\item[{\bf(i)}] Under the validity of the strict complementary condition \eqref{eq:strict_CP}, $\{x_k\}$ converges to $\overline{x}$ quadratically, i.e.,
\begin{equation*}
\limsup_{k\to \infty}\frac{\|x_{k}-\overline{x}\|}{\|x_{k-1}-\overline{x}\|^{2}}<\infty.
\end{equation*}
 
\item[{\bf(ii)}] If the strict complementary condition \eqref{eq:strict_CP} fails, then $\{x_k\}$ converges to $\overline{x}$ with the order $O(k^{-2})$.
\end{itemize}
\end{corollary}
\begin{proof} We know from Section~\ref{sec:2nd} that the strong strict saddle point property is fulfilled for $f=f_{OP}$. To verify (i), recall that if the strict complementary condition holds, then $\nabla f$ satisfies the metric subregularity property at $\overline{x}$ with respect to $\Theta$. Direct verifications show that $f$ is bounded from below by zero, that the set $\mathcal{L}(f(x_0))$ is bounded, and that the Hessian of $f$ is Lipschitz continuous on $\mathcal{F}=\mathcal{L}(f(x_0)) + \mathbb{B}_{\mathbb{R}^{2m}}(0,1)$. By setting $p=q=1$, we see that $\{x_k\}$ converges to $\overline{x} \in \Theta$ in a quadratic rate. Note also that $f$ satisfies the strong strict saddle point property, and thus $\overline{x}$ is indeed a global minimizer.

To prove (ii), observe that if the strict complementary condition fails, then $f$ satisfies the  KL property with the exponent $\theta=\frac{3}{4}$.  Therefore,  by setting $q=1$ in 
Proposition~\ref{corT6.2} we see that $\{x_k\}$ converges to $\overline{x} \in \Theta$ with the rate $O(k^{-\frac{(q+1)(1-\theta)}{(q+2)\theta-(q+1)}})=O(k^{-2})$. Similarly to (i), it follows that $\overline{x}$ is a global minimizer of $f$.
\end{proof}

\begin{remark}[\bf further comparison]\label{rem:comp} {\rm 
During the completion of this paper, we became aware of the preprint \cite{lpq} in which the authors examined the convergence rate of the high-order regularized method without momentum steps for composite optimization problems under the assumption that the Hessian of the smooth part is Lipschitz continuous labeled there as ``strict continuity." Recall that the main focus of our paper is on the development of generalized metric subregularity with applications to fast convergence analysis for general descent methods. This include, in particular, the high-order regularized methods with momentum steps. We also mention that the results obtained in \cite{lpq} cannot cover the setting of Section~\ref{sec:alg}, where we incorporate momentum steps and can also handle the case of H\"{o}lder continuous Hessians.}
\end{remark}\vspace*{-0.2in}

\section{Concluding Remarks}\label{conc}\vspace*{-0.05in}
 
The main contributions of this paper are summarized as follows. 

$\bullet$ We introduce a generalized metric subregularity condition, which serves as an extension of the standard metric subregularity and its H\"{o}lderian version while being strictly weaker than the uniform version (error bound condition) recently employed in the convergence analysis of the cubic regularized Newton method.  We then develop an abstract extended convergence framework that enables us to derive superlinear convergence towards specific target sets (such as the first-order and second-order stationary points) under the introduced generalized metric subregularity condition. Interestingly, this framework also extends the widely used KL convergence analysis, and the derived superlinear convergence rates of this new framework are sharp in the sense that the rate can be attained. 

$\bullet$  We provide easily verifiable sufficient conditions for the generalized metric subregularity with respect to the second-order stationary set under the strict saddle point property. We also demonstrate that several modern practical models appearing in machine learning, statistics, etc., enjoy the generalized metric subregularity property with respect to the second-order stationary set in explicit forms. This includes, e.g., objective functions that arise in the over-parameterized compressed sensing model, best rank-one matrix approximation, and generalized phase retrieval problems. 

$\bullet$ As a major application of generalized metric subregularity and abstract convergence analysis, we derive superlinear convergence results for the new high-order regularized Newton method with momentum. This improves the existing literature by weakening the regularity conditions and extending the analysis to include momentum steps. Furthermore, when applying the cubic regularized Newton methods with momentum steps to solve the over-parameterized compressed sensing model, we obtain global convergence with a quadratic local convergence rate towards a global minimizer under the strict complementarity condition.\vspace{0.05in} 

One important potential future research topic is to extend the abstract convergence framework to cover nonmonotone numerical algorithms. Another major goal of our future research is to extend the obtained second-order results from ${\cal C}^2$-smooth functions to the class of ${\cal C}^{1,1}$ functions \cite{m24}, i.e., continuously differentiable functions with Lipschitzian gradients, or to functions with a composite structure; see, e.g., \cite{NN24}. \vspace*{-0.2in}

\section{Proofs of Technical Lemmas}\label{appe}
\setcounter{equation}{0}\vspace*{-0.05in}

The main goal of this appendix is to give a detailed proof of 
Lemma~\ref{L3.1}. We split this proof into several technical steps (also called lemmas), which are of their own interest.

\begin{lemma}\label{L4.2} Under Assumption~{\rm\ref{A3.1}}, for any $x\in \mathcal{F}$ we have
\begin{equation}\label{4.9}
\langle \nabla f(x),x- y_{\sigma}(x)\rangle\geq 0.
\end{equation}
If furthermore $\sigma\geq \frac{2L_2}{q+2}$ and $f(x)\leq f(x_0)$, then $y_{\sigma}(x)\in \mathcal{L}(f(x))\subset \mathcal{F}$.
\end{lemma}
\begin{proof}
It readily follows from \eqref{4.5} that
\begin{equation}\label{4.8}
\langle \nabla f(x), y_{\sigma}(x)-x\rangle+ \langle \nabla^2 f(x)(y_{\sigma}(x)-x), y_{\sigma}(x)-x\rangle +\frac{\sigma}{q+1}\|y_{\sigma}(x)-x\|^{q+2}=0.
\end{equation}
Employing \eqref{4.6}, we obtain
\begin{equation*}
\langle \nabla^2 f(x)(y_{\sigma}(x)-x), y_{\sigma}(x)-x\rangle +\frac{\sigma}{q+1}\|y_{\sigma}(x)-x\|^{q+2}\geq 0,
\end{equation*}
and hence \eqref{4.9} is a direct consequence of \eqref{4.8}.
  
Let now $\sigma\geq \frac{2L_2}{q+2}$ and $f(x) \le f(x_0)$. To see the conclusion, we argue by contradiction and suppose that $y_{\sigma}(x)\notin \mathcal{F}$, which yields $\|x-y_{\sigma}(x)\|>0$. Denote
\begin{equation}\label{4.11}
y_\alpha:=x+\alpha(y_{\sigma}(x)-x), \quad \alpha\in [0,1],
\end{equation} 
and deduce from $\mathcal{L}(f(x_0) \subset {\rm int} (\mathcal{F})$ (by (2) in Assumption~\ref{A3.1}) and $f(x) \le f(x_0)$ that $y_0=x\in {\rm int}(\mathcal{F})$. So the number $\overline{\alpha}:=\min\{\alpha> 0\;|\;y_{\alpha}\in {\rm bd}( \mathcal{F})\}$ is well-defined. It follows that $\overline{\alpha}<1$ and that $y_{\alpha}\in \mathcal{F}$ for all $\alpha\in [0, \overline{\alpha}]$. Consequently, combining Lemma~\ref{L3.2} with  \eqref{4.9}--\eqref{4.11} gives us the estimates
\begin{equation*}
\begin{aligned}
f(y_\alpha)\leq &f(x)+\langle \nabla f(x), y_{\alpha}-x\rangle+ \frac{1}{2}\langle \nabla^2 f(x)(y_\alpha-x), y_\alpha-x\rangle +\frac{L_2}{(q+1)(q+2)}\|y_\alpha-x\|^{q+2}\\
\leq &f(x)+\alpha\langle \nabla f(x), y_{\sigma}(x)-x\rangle+ \frac{\alpha^2}{2}\langle \nabla^2 f(x)(y_{\sigma}(x)-x), y_{\sigma}(x)-x\rangle +\frac{\alpha^{q+2}\sigma}{2(q+1)}\|y_{\sigma}(x)-x\|^{q+2}\\
\leq &f(x)+\left(\alpha-\frac{\alpha^2}{2}\right)\langle \nabla f(x), y_{\sigma}(x)-x\rangle-\left(\frac{\sigma\alpha^2}{2(q+1)}-\frac{\alpha^{q+2}\sigma}{2(q+1)}\right)\|y_{\sigma}(x)-x\|^{q+2}\\
\leq &f(x)-\frac{\sigma\alpha^2(1-\alpha^q)}{2(q+1)}\|y_{\sigma}(x)-x\|^{q+2},
\end{aligned}
\end{equation*}
which imply in turn that $f(y_{\overline{\alpha}})<f(x)$. Therefore, $y_{\overline{\alpha}}\in{\rm int}(\mathcal{F})$, a contradiction  ensuring that $y_{\sigma}(x)\in \mathcal{F}$. In a similar way, we can check that $f(y_{\sigma}(x))\leq f(x)$.
\end{proof}

\begin{lemma}\label{L4.3}
Under Assumption~{\rm\ref{A3.1}}, for any $x\in \mathcal{F}$ with $y_{\sigma}(x)\in \mathcal{F}$ we have
\begin{equation}\label{4.14}
\|\nabla f(y_{\sigma}(x))\|\leq \frac{L_2+\sigma}{q+1}\|y_{\sigma}(x)-x\|^{q+1}.
\end{equation}
\end{lemma}
\begin{proof}
It follows from equation \eqref{4.5} that
\begin{equation*}
\| \nabla f(x)+  \nabla^2 f(x)(y_{\sigma}(x)-x)\|=\frac{\sigma}{q+1}\|y_{\sigma}(x)-x\|^{q+1}.
\end{equation*}
On the other hand, we deduce from \eqref{4.27} that
\begin{equation*}
\|\nabla f(y_{\sigma}(x))-\nabla f(x)-\nabla^2f(x)(y_{\sigma}(x)-x)\|\leq \frac{L_2}{q+1}\|y_{\sigma}(x)-x\|^{q+1}.
\end{equation*}
Combining these two relations verifies the desired inequality \eqref{4.14}.
\end{proof}

\begin{lemma}
Under Assumption~{\rm\ref{A3.1}}, for any $x\in \mathcal{F}$ we have the estimates
\begin{equation}\label{4.18}
\overline{f}_{\sigma}(x)\leq \min_{y\in \mathcal{F}}\left[f(y)+\frac{L_2+\sigma}{(q+1)(q+2)}\|y-x\|^{q+2}\right],
\end{equation}
\begin{equation}\label{4.19}
f(x)-\overline{f}_{\sigma}(x)\ge\frac{q\sigma}{2(q+1)(q+2)}\|y_{\sigma}(x)-x\|^{q+2}.
\end{equation}
Moreover, if $\sigma\geq L_2$  and $f(x)\leq f(x_0)$, then $y_{\sigma}(x)\in \mathcal{F}$ and
\begin{equation}\label{4.20}
f(y_{\sigma}(x))\leq \overline{f}_{\sigma}(x).
\end{equation}
\end{lemma}

\begin{proof}
Employing the lower bound in \eqref{4.33}, for any $x,y\in \mathcal{F}$, we get
\begin{equation*}
f(x)+\langle \nabla f(x), y-x\rangle+ \frac{1}{2}\langle \nabla^2 f(x)(y-x),y-x\rangle\leq f(y)+\frac{L_2}{(q+1)(q+2)}\|y-x\|^{q+2},
\end{equation*}
and so inequality \eqref{4.18} follows from the construction of $\overline{f}_{\sigma}(x)$. 
%The estimate in \eqref{4.19} can be checked similarly by using \eqref{4.27}.
Furthermore, we deduce from $y_{\sigma}(x)\in \mathop{\arg\min}_{y \in \mathbb{R}^m}\overline{f}_{\sigma}(x)$,   \eqref{4.8}, and Lemma~\ref{L4.2} that
\begin{equation*}
\begin{aligned}
f(x)-\overline{f}_{\sigma}(x)
&=\langle \nabla f(x), x-y_{\sigma}(x)\rangle- \frac{1}{2}\langle \nabla^2 f(x)(y_{\sigma}(x)-x),y_{\sigma}(x)-x\rangle-\frac{\sigma}{(q+1)(q+2)}\|y_{\sigma}(x)-x\|^{q+2}\\
&=\frac{1}{2}\langle \nabla f(x), x-y_{\sigma}(x)\rangle +\frac{q\sigma}{2(q+1)(q+2)}\|y_{\sigma}(x)-x\|^{q+2}\ge\frac{q\sigma}{2(q+1)(q+2)}\|y_{\sigma}(x)-x\|^{q+2}.
\end{aligned}
\end{equation*}
Finally, assuming $\sigma\geq L_2$ and $f(x)\le f(x_0)$ gives us by 
Lemma~\ref{L4.2} that $y_{\sigma}(x)\in \mathcal{F}$, and thus \eqref{4.20} follows from the upper bound in \eqref{4.33}.
\end{proof}

\begin{lemma}\label{L4.5}
Under Assumption~{\rm\ref{A3.1}}, for any $x\in \mathcal{F}$ with $y_{\sigma}(x)\in \mathcal{F}$ we have
\begin{equation}\label{618}
\lambda_{\min}\Big( \nabla^2f(y_{\sigma}(x))\Big)\geq -\left(\frac{\sigma}{q+1}+L_2\right)\|y_{\sigma}(x)-x\|^{q}.
\end{equation}
\end{lemma}
\begin{proof}
The H\"{o}lder continuity of $\nabla^2f(x)$ and inequality \eqref{4.6} tell us that
\begin{equation*}
\begin{aligned}
\nabla^2f(y_{\sigma}(x))\succeq \nabla^2f(x)-L_2\|y_{\sigma}(x)-x\|^{q}I_m\succeq -\left(\frac{\sigma}{q+1}+L_2\right)\|y_{\sigma}(x)-x\|^{q}I_m,
\end{aligned}
\end{equation*}
which clearly yields the estimate in \eqref{618}.
\end{proof}

\begin{lemma}
Let Assumption~{\rm\ref{A3.1}} hold, and let the sequences $\{x_k\}$ and $\{\widehat x_k\}$ be taken from Algorithm~{\rm 1} with $\overline{\sigma} >\frac{2L_2}{q+2}$. Then for all $k\in\mathbb N$ we have the estimates
\begin{equation}\label{6.18}
f(\widehat x_{k+1})-f(x_k)\leq -\frac{(q+2)\overline{\sigma}-2L_2}{2(q+1)(q+2)}\|\widehat x_{k+1}-x_k\|^{q+2},
\end{equation}
\begin{equation}\label{6.19}
\sum_{i=0}^{k}\|\widehat x_{k+1}-x_k\|^{q+2}\leq \frac{2(q+1)(q+2)}{(q+2)\overline{\sigma}-2L_2} (f(x_0)-f^*),
\end{equation}
\begin{equation}\label{6.20}
\|\nabla f(\widehat x_{k+1})\|\leq \frac{3L_2}{q+1}\|\widehat x_{k+1}-x_k\|^{q+1},
\end{equation}
\begin{equation}\label{6.21}
\lambda_{\min}(\nabla^2f(\widehat x_{k+1}))\geq -\frac{(q+3)L_2}{q+1}\|\widehat x_{k+1}-x_k\|^{q}.
\end{equation}
\end{lemma}
\begin{proof}
First we verify \eqref{6.18}. It follows from Lemma~\ref{L3.2}, \eqref{A89}, and \eqref{4.19} that
\begin{equation*}
\begin{aligned}
&f(\widehat x_{k+1})\le f(x_k)+\langle \nabla f(x_k), \widehat x_{k+1}-x_k\rangle\\
&+ \frac{1}{2}\langle \nabla^2 f(x_k)(\widehat x_{k+1}-x_k), \widehat x_{k+1}-x_k\rangle+\frac{L_2}{(q+1)(q+2)}\|\widehat x_{k+1}-x_k\|^{q+2}\\
&=f(x_k)+\langle \nabla f(x_k), \widehat x_{k+1}-x_k\rangle+ \frac{1}{2}\langle \nabla^2 f(x_k)(\widehat x_{k+1}-x_k), \widehat x_{k+1}-x_k\rangle\\
&+\frac{\sigma_k}{(q+1)(q+2)}\|\widehat x_{k+1}-x_k\|^{q+2}+\frac{L_2-\sigma_k}{(q+1)(q+2)}\|\widehat x_{k+1}-x_k\|^{q+2}\\
&\leq f(x_k)-\frac{q\sigma_k}{2(q+1)(q+2)}\|\widehat x_{k+1}-x_k\|^{q+2}+\frac{L_2-\sigma_k}{(q+1)(q+2)}\|\widehat x_{k+1}-x_k\|^{q+2}\\
&=f(x_k)-\frac{(q+2)\sigma_k-2L_2}{2(q+1)(q+2)}\|\widehat x_{k+1}-x_k\|^{q+2}\le f(x_k)-\frac{(q+2)\overline{\sigma}-2L_2}{2(q+1)(q+2)}\|\widehat x_{k+1}-x_k\|^{q+2},
\end{aligned}
\end{equation*}
which justifies \eqref{6.18}. To prove \eqref{6.19}, we use \eqref{5.27} and \eqref{6.18} to get
\begin{equation*}
\begin{aligned}
f(x_0)-f^*\geq &\sum_{i=0}^{k}[f(x_i)-f(x_{i+1})]\geq\sum_{i=0}^{k}[f(x_i)-f(\widehat x_{i+1})]\\
\geq & \sum_{i=0}^{k}\frac{(q+2)\overline{\sigma}-2L_2}{2(q+1)(q+2)}\|\widehat x_{i+1}-x_i\|^{q+2},
\end{aligned}
\end{equation*}
which yields \eqref{6.19}. The inequalities in \eqref{6.20} and \eqref{6.21} follow from Lemmas~\ref{L4.3} and \ref{L4.5}, respectively.
\end{proof}

The next estimates follow from \eqref{5.27} and \eqref{5.026} being useful for our subsequent analysis:
\begin{equation}\label{6.26}
\begin{aligned}
\|x_{k+1}-\widehat x_{k+1}\|\leq &\max\big\{ \|\widehat x_{k+1}-\widehat x_{k+1}\|, \|\widetilde x_{k+1}-\widehat x_{k+1}\|\big\}\\
=&\|\widetilde x_{k+1}-\widehat x_{k+1}\|\leq \beta_{k+1} \|\widehat x_{k+1}-\widehat x_{k}\|.
\end{aligned}
\end{equation}

\begin{lemma}\label{L6.6} Under Assumption~{\rm\ref{A3.1}}, the sequence $\{\widehat x_k\}$ generated by Algorithm~{\rm 1} satisfies
\begin{equation*}
\|\widehat x_{k+1}-\widehat x_{k}\|\leq \frac{1}{1-\zeta} \left[\frac{2(q+1)(q+2)}{(q+2)\overline{\sigma}-2L_2} (f(x_0)-f^*)\right]^{\frac{1}{q+2}}=:\gamma.
\end{equation*}
\end{lemma}
\begin{proof}
For any index $i\geq 1$, we deduce from \eqref{6.26} and the condition 
$\beta_{k+1}\le\zeta$ for all $k\geq 0$ that
\begin{equation}\label{6.28}
\begin{aligned}
\|\widehat x_{i+1}-\widehat x_{i}\|\leq&   \|\widehat x_{i+1}- x_i\|+  \|x_i-\widehat x_i\|
\le \|\widehat x_{i+1}- x_i\|+  \beta_i\|\widehat x_i-\widehat x_{i-1}\|
\le\|\widehat x_{i+1}- x_i\|+  \zeta\|\widehat x_i-\widehat x_{i-1}\|.
\end{aligned}
\end{equation}
Recursively applying \eqref{6.28}, we deduce from $\widehat x_0=x_0$ that
\begin{equation*}
\begin{aligned}
\|\widehat x_{k+1}-\widehat x_k\|\leq &\zeta^k  \|\widehat x_{1}-\widehat x_{0}\|+\sum_{i=1}^k\zeta^{k-i}\|\widehat x_{i+1}- x_i\|
\le \sum_{i=0}^k\zeta^{k-i}\|\widehat x_{i+1}- x_i\|
\end{aligned}
\end{equation*}
whenever $k\in\mathbb N$. Then it follows from \eqref{6.19} that
\begin{equation*}
\begin{aligned}
\|\widehat x_{k+1}-\widehat x_k\|&\leq\max_{i\in\{0, \cdots,k\}}\|\widehat x_{i+1}- x_i\|\sum_{i=0}^k\zeta^{k-i}
\le \frac{1}{1-\zeta}\max_{i\in\{0, \cdots,k\}}\|\widehat x_{i+1}- x_i\|\\
&\le \frac{1}{1-\zeta}\left[\frac{2(q+1)(q+2)}{(q+2)\overline{\sigma}-2L_2} (f(x_0)-f^*)\right]^{\frac{1}{q+2}},
\end{aligned}
\end{equation*}
which justified the claim of the lemma.
\end{proof}

Now we are ready to finalize the {\bf Proof of Lemma \ref{L3.1}}:\\[1ex]
To verify assertion (i) therein, it suffices to show that $\{f(x_k)\}$ is a decreasing sequence having a lower bound. By Assumption~\ref{A3.1}, $f$ is bounded from below and so is $\{f(x_k)\}$. Employing \eqref{5.27} and \eqref{6.18} yields
\begin{equation}\label{6.31}
f(x_{k+1})\leq f(\widehat x_{k+1})\leq f(x_k)-\frac{(q+2)\overline{\sigma}-2L_2}{2(q+1)(q+2)}\|\widehat x_{k+1}-x_k\|^{q+2}
\end{equation}
and allows us to deduce that $\{f(x_k)\}$ converges to some $v$.\\[1ex]
To verify (ii), we begin with using \eqref{6.26}, $\beta_{k+1}\leq \|\widehat x_{k+1}-x_k\|$, and Lemma~\ref{L6.6} to get
\begin{equation*}
\begin{aligned}
&\|x_{k+1}-x_k\|\leq\|x_{k+1}-\widehat x_{k+1}\|+\|\widehat x_{k+1}-x_k\|
\le\beta_{k+1}\|\widehat x_{k+1}-\widehat x_k\|+\|\widehat x_{k+1}-x_k\|\\
&\leq\|\widehat x_{k+1}-x_k\|(\|\widehat x_{k+1}-\widehat x_k\|+1)
\leq(1+\gamma)\|\widehat x_{k+1}-x_k\|,
\end{aligned}
\end{equation*}
which tells us together with \eqref{6.19} that
\begin{equation}\label{6.33}
\lim_{k\to \infty}\|\widehat x_{k+1}-x_k\|=0.
\end{equation}
Combining further Assumption~\ref{A3.1} and Lemma~\ref{L6.6} with \eqref{5.025}, \eqref{6.26}, and \eqref{6.20}, we arrive at
\begin{equation}\label{6.34}
\begin{aligned}
&\|\nabla f(x_{k+1})\|\le\|\nabla f(\widehat x_{k+1})\|+\|\nabla f(\widehat x_{k+1})-\nabla f(x_{k+1})\|\\
&\leq\|\nabla f(\widehat x_{k+1})\|+L_1\|x_{k+1}-\widehat x_{k+1}\|
\leq\|\nabla f(\widehat x_{k+1})\|+L_1\beta_{k+1} \|\widehat x_{k+1}-\widehat x_{k}\|\\
&\le\|\nabla f(\widehat x_{k+1})\|(1+L_1\|\widehat x_{k+1}-\widehat x_{k}\|)\|\le\frac{3(1+\gamma L_1)L_2}{q+1}\|\widehat x_{k+1}-x_k\|^{q+1},
\end{aligned}
\end{equation}
and therefore verify assertion (ii).\\[1ex]
To verify the final assertion (iii), note that \eqref{6.31} yields $\{x_k\}\subset \mathcal{L}(f(x_0))$. Recalling that $\mathcal{L}(f(x_{k_0}))$ is bounded for some $k_0\in\mathbb N$, we conclude that $\{x_k\}$ is bounded as well, which ensures that $\Omega\ne\emptyset$. Pick any $\overline{x}\in \Omega$ and observe that justifying \eqref{5.53} requires to bound the minimum eigenvalue of the Hessian. Recalling that $|\lambda_{\min}(C)-\lambda_{\min}(D)|\le\|C-D\|$ holds for any real symmetric matrices $C$ and $D$, we have
\begin{equation*}
\begin{aligned}
\lambda_{\min}(\nabla^2f( x_{k+1}))\geq \lambda_{\min}(\nabla^2f(\widehat x_{k+1}))-\|\nabla^2 f(x_{k+1})-\nabla^2 f(\widehat x_{k+1})\|.
\end{aligned}
\end{equation*}
This together with Assumption~\ref{A3.1}, Lemma~\ref{L6.6}, \eqref{6.21}, \eqref{6.26}, and $\beta_{k+1}\leq \|\widehat x_{k+1}-x_k\|$ tells us that
\begin{equation*}
\begin{aligned}
&\lambda_{\min}(\nabla^2f( x_{k+1}))\ge\lambda_{\min}(\nabla^2f(\widehat x_{k+1}))-\|\nabla^2 f(x_{k+1})-\nabla^2 f(\widehat x_{k+1})\|\ge\lambda_{\min}(\nabla^2f(\widehat x_{k+1}))-L_2 \|x_{k+1}-\widehat x_{k+1}\|^q\\
&\ge\lambda_{\min}(\nabla^2f(\widehat x_{k+1}))-L_2 \beta_{k+1}^q \|\widehat x_{k+1}-\widehat x_{k}\|^q\ge\lambda_{\min}(\nabla^2f(\widehat x_{k+1}))-L_2\gamma^q \|\widehat x_{k+1}-x_k\|^q\\
&\ge-\frac{(q+3)L_2}{q+1}\|\widehat x_{k+1}-x_k\|^{q}-L_2\gamma^q \|\widehat x_{k+1}-x_k\|^q
=-\left(\frac{q+3}{q+1}+\gamma^q\right)L_2\|\widehat x_{k+1}-x_k\|^{q}.
\end{aligned}
\end{equation*}
Unifying the latter with \eqref{6.33} and \eqref{6.34} leads us to the inequalities
\begin{equation*}
\|\nabla f(\overline{x})\|\leq \limsup_{k\to \infty} \|\nabla f(x_{k+1})\|\leq 0\quad \mbox{and}\quad   \lambda_{\min}(\nabla^2f(\overline{x}))\geq \liminf_{k\to \infty}\lambda_{\min}(\nabla^2f( x_{k+1}))\geq 0,
\end{equation*}
which show that $\nabla f(\overline{x})=0$ and $\nabla^2 f(\overline{x})\succeq 0$. Employing finally Lemma~\ref{L3.1}(i), we obtain $f(\overline{x})=v$ and therefore complete the proof of the lemma.
\end{document}